\documentclass[12pt]{amsart}

\usepackage{amssymb,latexsym}

\usepackage{enumerate}

 \usepackage[french,english]{babel}
 
 \usepackage{xcolor}

\makeatletter

\@namedef{subjclassname@2010}{

  \textup{2010} Mathematics Subject Classification}

\makeatother
\newtheorem{thm}{Theorem}[section]

\newtheorem{Pro}[thm]{Proposition}
\newtheorem{cor}[thm]{Corollary}
\newtheorem{lem}[thm]{Lemma}
\newtheorem{pro}[thm]{Proposition}
\theoremstyle{definition}
\newtheorem{defin}[thm]{Definition}
\newtheorem{rem}[thm]{Remark}

\numberwithin{equation}{section}

\DeclareMathOperator{\rank}{Rank}
\DeclareMathOperator{\sing}{Sing}
\DeclareMathOperator{\Swan}{Swan}
\DeclareMathOperator{\Drop}{Drop}
\DeclareMathOperator{\Fr}{Fr}
\DeclareMathOperator{\Tr}{Tr}
\DeclareMathOperator{\FT}{FT}

\DeclareMathOperator{\Sp}{Sp}
\DeclareMathOperator{\USp}{USp}
\DeclareMathOperator{\SL}{SL}

\DeclareMathOperator{\Kl}{Kl}

\DeclareMathOperator{\PGL}{PGL}
\newcommand{\X}{\mathbb{X}}
\newcommand{\Y}{\mathbb{Y}}
\newcommand{\ex}{\mathbb{E}}
\newcommand{\E}{\mathcal{E}}
\newcommand{\re}{\textup{Re}}
\newcommand{\im}{\textup{Im}}
\newcommand{\pr}{\mathbb{P}}
\newcommand{\f}{f_{\mathbb{X}}}
\newcommand{\B}{\textup{Bi}_p}
\newcommand{\ta}{\widehat{\ph_a}}

\newcommand{\G}{\mathcal{G}}
\newcommand{\F}{\mathbb{F}_p}
\newcommand{\Fs}{\mathbb{F}_p^{\times}}
\newcommand{\Fn}{\mathbb{F}_{p^{n}}}
\newcommand{\M}{\mathcal{M}}

\newcommand{\ph}{\varphi}

\newcommand{\newabstract}[1]{%
  \par\bigskip
  \csname otherlanguage*\endcsname{#1}%
  \csname captions#1\endcsname
  \item[\hskip\labelsep\scshape\abstractname.]
}

\begin{document}

\baselineskip=17pt

\title[The distribution of the maximum of partial sums of trace functions]{The distribution of the maximum of partial sums of Kloosterman sums and other trace functions}

\author{Pascal Autissier}
\address{I.M.B., Universit\'e de Bordeaux, 
351, cours de la Lib\'eration, 
33405 TALENCE, France}

\email{Pascal.Autissier@math.u-bordeaux.fr}

\author{Dante Bonolis}
\address{ETH Zurich,
Department of Mathematics,
HG J 14.4,
R\"amistrasse 101,
8092 Zurich,
Switzerland}

\email{dante.bonolis@math.ethz.ch}

\author{Youness Lamzouri}
\address{Institut \'Elie Cartan de Lorraine, Universit\'e de Lorraine, BP 70239, 54506 Vandoeuvre-l\`es-Nancy Cedex, France
\\ and Department of Mathematics and Statistics,
York University,
4700 Keele Street,
Toronto, ON,
M3J1P3
Canada}

\email{youness.lamzouri@univ-lorraine.fr}

\date{}

\begin{abstract} In this paper, we investigate the distribution of the maximum of partial sums of families of $m$-periodic complex valued functions satisfying certain conditions. We obtain precise uniform estimates for the distribution function of this maximum in a near optimal range. Our results apply to partial sums of Kloosterman sums and other families of $\ell$-adic trace functions, and are as strong as those obtained by Bober, Goldmakher, Granville and Koukoulopoulos for character sums. In particular, we improve on the recent work of the third author for Birch sums. However, unlike character sums, we are able to construct families of $m$-periodic complex valued functions which satisfy our conditions, but for which the P\'olya-Vinogradov inequality is sharp.  
\end{abstract}


\thanks{The third author is partially supported by a Discovery Grant from the Natural Sciences and Engineering Research Council of Canada.}

\maketitle

\section{Introduction}
Let $m\geq 2$ be an integer, and $\ph: \mathbb{Z}/m\mathbb{Z} \to \mathbb{C}$ a complex valued function which we extend to an $m$-periodic function $\ph: \mathbb{Z} \to \mathbb{C}$. An important problem in analytic number theory is to obtain non-trivial estimates for the quantity 
$$ \M(\ph):= \max_{x<m} \left|\sum_{0\leq n\leq x} \ph(n)\right|.$$ 
The special case where $\ph=\chi$ is a Dirichlet character modulo $m$ has been extensively studied over the last century, going back to the classical inequality proved by P\'olya and Vinogradov in 1918:
$$ \M(\chi)\ll \sqrt{m} \log m.$$  
A straightforward generalization of this bound for a general $m$-periodic complex valued function $\ph$ gives 
\begin{equation}\label{PV}
\M(\ph)\ll ||\widehat{\ph}||_{\infty} \sqrt{m}\log m,
\end{equation}
where $\widehat{\ph}: \mathbb{Z}\to \mathbb{C}$ is the normalized discrete Fourier transform of $\ph$, defined by
$$ \widehat{\ph} (h) = \frac{1}{\sqrt{m}} \sum_{n \ (\textup{mod } m)} \ph(n) e_m(hn),$$
where here and throughout we let $e(z):=\exp(2\pi i z)$, and $e_m(z):=e(z/m)$ is the standard additive character modulo $m$. To see this, consider the discrete Plancherel formula 
\begin{equation}\label{Planche}
\sum_{0\leq n\leq x} \ph(n)=\sum_{-m/2<h\leq m/2}\overline{\gamma_m(h;x)} \widehat{\ph}(h),
\end{equation}
where $$\gamma_m(h;x):= \frac{1}{\sqrt{m}} \sum_{0\leq n\leq x}e_m\left(nh\right) $$
are the Fourier coefficients modulo $m$ of the characteristic function of the interval $[0, x]$. The P\'olya-Vinogradov bound \eqref{PV} follows by using the elementary estimate (see for example page 1501 of \cite{KoSa}) 
\begin{equation}\label{Planche2}
\frac{1}{\sqrt{m}}\gamma_m(h;x)= \frac{e_m\left(xh\right)- 1}{2\pi i h} +O\left(\frac{1}{m}\right),
\end{equation}
which holds uniformly for $1\leq |h|\leq m/2$. 

We shall only consider those $\ph$ for which the Fourier transform $\widehat{\ph}$ is uniformly bounded (this includes primitive Dirichlet characters), which in view of the P\'olya-Vinogradov bound \eqref{PV} gives 
\begin{equation}\label{PV2}
 \M(\ph)\ll  \sqrt{m}\log m.
 \end{equation} In the case of character sums, Montgomery and Vaughan \cite{MV} proved that this bound is not optimal conditionally on the generalized Riemann hypothesis GRH. Indeed, they showed that assuming GRH we have 
$$ \M(\chi)\ll \sqrt{m}\log\log m,$$  for all non-principal Dirichlet characters $\chi\pmod m$.  This last bound is in fact optimal in view of an old result of Paley \cite{Pa} who showed that $\M(\chi_m)\gg \sqrt{m}\log\log m$ for infinitely many $m$, where $\chi_m$ is the quadratic character modulo $m$.  

Recently, Bober, Goldmakher, Granville and Koukoulopoulos \cite{BGGK} investigated the distribution of $\M(\chi)$ over non-principal characters $\chi$ modulo a large prime $q$.  If we denote by 
$\Phi_{\textup{char}}(V)$ the proportion of non-principal characters $\chi\bmod q$ for which $\M(\chi)/\sqrt{q}>V$, then the main result of \cite{BGGK} states that 
for $C\leq V\leq C_0\log \log q -C$ (where $C$ is an absolute constant), one has 
\begin{equation}\label{BGGK}
 \Phi_{\textup{char}}(V) = \exp\left(-\frac{e^{V/C_0+O(1)}}{V}\right),
 \end{equation}
where $C_0=e^{\gamma}/\pi$, and $\gamma$ is the Euler-Mascheroni constant.

Building on the work of Kowalski and Sawin \cite{KoSa},  Lamzouri \cite{La} investigated a similar question for the partial sums of certain exponential sums. For a prime $p\ge 3$ the Birch sum associated to $a \in \F$ is the following normalized complete cubic exponential sum 
$$
\B(a):=\frac{1}{\sqrt{p}}\sum_{n\in \F} e_p(n^3+an).
$$
These sums were first considered by Birch \cite{Bi} who conjectured that $\B(a)$ becomes equidistributed according to the Sato-Tate measure as $a$ varies in $\Fs$ and $p\to\infty$. This conjecture was subsequently proved by Livn\'e in \cite{Li}.
Let $\ph_a(n)= e_p(n^3+an)$ and define
$$\Phi_{\textup{Bi}}(V)= \frac{1}{p-1}\left|\left\{ a\in \Fs: \frac{\M(\ph_a)}{\sqrt{p}}>V\right\}\right|.$$
Lamzouri \cite{La} proved that for $V$ in the range $1\ll V\leq (2/\pi)\log\log p-2\log\log \log p$, we have
\begin{equation}\label{BirchBounds}
 \exp\left(-\exp\left(\frac{\pi}{2} V+O(1)\right)\right) \leq \Phi_{\textup{Bi}}(V) \leq \exp\left(- \exp\left(\left(\frac{\pi}{2}-\delta\right)V+O(1)\right)\right)
\end{equation} where 
$ \delta= \frac{4\pi-\pi^2}{2\pi +8}=0.18880...$. He also conjectured that the lower bound corresponds to the true order of magnitude for $\Phi_{\textup{Bi}}(V)$. The techniques are different in this setting, due to the lack of multiplicativity for these exponential sums. Indeed, in the case of character sums, Bober, Goldmakher, Granville and Koukoulopoulos \cite{BGGK} exploit the relation with $L$-functions and smooth numbers, while ingredients from algebraic geometry and notably Deligne's equidistribution theorem play a central role in \cite{La}.

Lamzouri also showed that the lower bound in \eqref{BirchBounds} holds for the maximum of partial sums of Kloosterman sums. The normalized classical Kloosterman sums are defined by 
$$
\textup{Kl}_p(a,b):=\frac{1}{\sqrt{p}}\sum_{n\in \Fs} e_p(an+b\overline{n}),
$$
where $\overline{n}$ denotes the multiplicative inverse of $n$ modulo $p$. Similarly to Birch sums, Katz \cite{Kat88} proved that $\text{Kl}_p(a, 1)$ becomes equidistributed  according to the Sato-Tate measure as $a$ varies in $\Fs$ and $p\to\infty$. Let $\ph_{(a, b)}(n)=e_p(an+b\overline{n}).$
The method of \cite{La} allows one to prove that in the range  $1\ll V\leq (2/\pi)\log\log p-2\log\log \log p$ we have
$$ \Phi_{\textup{Kl}}(V):=\frac{1}{(p-1)^2}\left|\left\{ (a, b)\in \Fs\times\Fs: \frac{\M(\ph_{(a, b)})}{\sqrt{p}}>V\right\}\right|\geq \exp\left(-\exp\left(\frac{\pi}{2} V+O(1)\right)\right).$$
  However, the argument is not strong enough to yield an upper bound for the distribution function $\Phi_{\textup{Kl}}(V)$ in this case, since it relies on strong bounds for short sums of exponential sums, which are not currently known for Kloosterman sums.

In this paper, we prove Lamzouri's conjecture for the maximum of partial sums of Birch and Kloosterman sums, obtaining estimates for their distribution functions that are as strong as \eqref{BGGK} for character sums.  
We also obtain analogous results for families of periodic functions which satisfy certain hypotheses (see Theorem \ref{Main} below). A corollary of our main theorem is the following result.

\begin{cor}\label{BirchKlooster}
Let $p$ be a large prime. There exists a constant $C$ such that for all real numbers $C\le V\leq (2/\pi)(\log\log p-2\log\log\log p)-C$ we have 
$$  \Phi_{\textup{Kl}}(V) = \exp\left(- \exp\left(\frac{\pi}{2}V+O(1)\right)\right).
$$
The same estimate also holds for $\Phi_{\textup{Bi}}(V).$
\end{cor}

There are two new ingredients that allow us to prove Lamzouri's conjecture. The first is a non-trivial upper bound for the fourth moment of the maximum over all intervals $I$ of length $|I|\leq p^{1/2+\varepsilon}$ with $\varepsilon>0$ (intervals at the edge of the P\'olya-Vinogradov range), of short sums of Kloosterman sums over $I$ (see Lemma \ref{BoundMomentShort} for a general result of this type). This allows us to obtain the analogue of \eqref{BirchBounds} for $\Phi_{\textup{Kl}}(V)$. The second ingredient is a precise asymptotic formula for the maximum of a certain ``random'' exponential sum (see Theorem \ref{AsympG} below), which we use to replace the constant $\pi/2-\delta$ by $\pi/2$ in the upper bound of \eqref{BirchBounds}, thus proving Lamzouri's conjecture.

\subsection{A general result for the maximum of partial sums of $m$-periodic functions} We shall consider families of periodic functions $\mathcal{F}=\{\ph_a\}_{a\in \Omega_m}$, where $\Omega_m$ is a non-empty finite set, and for each $a\in \Omega_m$, $\ph_a: \mathbb{Z}\to \mathbb{C}$ is $m$-periodic and its Fourier transform $\widehat{\ph}_a$ is real-valued and uniformly bounded. 
For a positive real number $V$, we define
$$ \Phi_{\mathcal{F}}(V):= \frac{1}{|\Omega_m|} \left|\left\{ a\in \Omega_m: \frac{\M(\ph_a)}{\sqrt{m}}>V\right\}\right|.$$
We will obtain precise uniform estimates for this distribution function, assuming that our family $\mathcal{F}$ satisfies certain hypotheses, which are mainly related to the distribution of the Fourier transform $\widehat{\ph}_a$. Such assumptions will be verified by several important functions in analytic number theory, which arise naturally in applications and originate in the deep work of Deligne and others from algebraic geometry. These functions correspond to certain \emph{Frobenius trace functions} modulo $m$, and their analytic properties have been investigated by several authors, and notably in a series of recent works by
Fouvry, Kowalski, and Michel \cite{FKM14b}, \cite{FKM15}, \cite{FKM15b}, \cite{FKM19}, Fouvry,  Kowalski, Michel, Raju, Rivat, and Soundararajan \cite{FKMRRS}, Kowalski and Sawin \cite{KoSa}, and Perret-Gentil \cite{Per17}. In particular, these include the families of trace functions $\mathcal{F}_{\text{Bi}}=\{e_p(n^3+an)\}_{a\in \Fs}$ and $\mathcal{F}_{\text{Kl}}=\{e_p(an+b\overline{n})\}_{(a, b)\in \Fs\times \Fs}$, which give rise to partial sums of Birch and Kloosterman sums respectively.  More specifically, let $\mathcal{F}=\{\ph_a\}_{a\in \Omega_m}$ be a family of $m$-periodic complex valued functions, and consider the following assumptions:

\newpage

\textbf{Assumption 1. Uniform boundedness}

We have $\max_{a\in \Omega_m}||\ph_a||_{\infty}\ll 1,$ where the implied constant is independent of $m$.  

\smallskip

\textbf{Assumption 2. Support of the Fourier transform}

There exists an absolute constant $N>0$ such that for all $a\in \Omega_m$ and $h\in \mathbb{Z}/m\mathbb{Z}$ we have $\ta(h)\in [-N, N]$.

\smallskip

 \textbf{Assumption 3. Joint distribution of the Fourier transform}

There exists a sequence of I.I.D. random variables $\{\X(h)\}_{h\in \mathbb{Z}^*}$ supported on $[-N, N]$, and absolute constants $\eta \ge 1/2$ and $C_1>1$, such that for all positive integers $k\leq \log m/\log\log m$, and all $k$-uples $(h_1, \dots, h_k)\in (-m/2, m/2]^k$ with $h_{i}\neq 0$ for $i=1,...,k$ we have 
$$ 
\frac{1}{|\Omega_m|} \sum_{a\in \Omega_m} \widehat{\ph_a}(h_1) \cdots \widehat{\ph_a}(h_k)= \ex\left(\X(h_1) \dots \X(h_k)\right)+ O\left(\frac{C_1^k}{m^{\eta}}\right).
$$
Furthermore, if we let $\X$ be a random variable with the same distribution as the $\X(h)$, then $\X$ verifies the following conditions:  
\begin{itemize}
\item[3a.] There exists a positive constant $A$ such that for all $\varepsilon\in(0,1]$ we have $\pr(\X> N-\varepsilon) \gg \varepsilon^A$, and $\pr(\X<- N+\varepsilon) \gg \varepsilon^A$.
\item[3b.] For all integers $\ell\geq 0$ we have  
$\ex\left(\X^{2\ell+1}\right)=0 .$
\end{itemize}

\textbf{Assumption 4. Strong bounds for short sums on average}

There exist absolute constants $\alpha \geq 1$, and $0<\delta<1/2$ such that for any interval $I$ of length $|I|\leq m^{1/2+\delta}$, one has
\[
\frac{1}{|\Omega_{m}|}\sum_{a\in\Omega_{m}}\Big|\frac{1}{\sqrt{m}}\sum_{n\in I}\ph_{a}(n)\Big|^{\alpha}\ll m^{-1/2-\delta}.
\]

\smallskip

 Our main result is the following theorem.
\begin{thm}\label{Main}
Let $m$ be large, and $\mathcal{F}=\{\ph_a\}_{a\in \Omega_m}$ be a family of $m$-periodic complex valued functions satisfying one of the following subsets of the above assumptions:
\begin{itemize}
\item[A.] Assumption 2 and Assumption 3 with $\eta>1$.
\item[B.] Assumptions 1, 2, and Assumption 3 with $1/2<\eta\leq1$.
\item[C.] Assumptions 1, 2, 4, and Assumption 3 with $\eta=1/2$.  
\end{itemize} 
Then there exists a constant $B$ such that for all real numbers $B\le V\leq (N/\pi)(\log\log m-2\log\log\log m)-B$ we have 
\begin{equation}\label{MainEstimate}
 \Phi_{\mathcal{F}}(V) = \exp\left(- \exp\left(\frac{\pi}{N}V+O(1)\right)\right).
\end{equation}
\end{thm}
\begin{rem}
Case C) is the most interesting and difficult case of Theorem \ref{Main}. In particular, all the examples of trace functions we consider (including Kloosterman sums, see Corollaries \ref{GBirch}, \ref{GKloo} and \ref{HypKloo}) fall into this case. For these examples, the saving of $\sqrt{m}$ in the error term of Assumption 3 follows from Deligne's equidistribution theorem. 
\end{rem}
\begin{rem}\label{ShortKoWa}
Assumption 4 was first considered by Kowalski and Sawin \cite{KoSa} but for a different purpose. The authors of \cite{KoSa} investigated \emph{Birch and Kloosterman paths}, which are the polygonal paths formed by linearly interpolating the partial sums of Birch and Kloosterman sums. They used Assumption 4 to establish a weak-compactness property known as \emph{tightness}, which was necessary in order to show that the processes obtained from Birch and Kloosterman paths converge in law (in the Banach space $C[0, 1]$) to a random Fourier series (which is the series inside the absolute value in \eqref{RandomKoSa} below). In our case, we found a new argument that allows us to use Assumption 4 (which holds for Kloosterman sums) instead of strong point wise bounds for short sums of exponential sums, which were needed in \cite{La}. 
\end{rem}
\begin{rem}\label{SizeOmega}
One can wonder whether a condition on the size of $\Omega_m$ is necessary to prove Theorem \ref{Main}. In fact, such a condition is implicitly contained in Assumptions 2 and 3. More specifically, we show in Lemma \ref{SizeOmegam} below that if $\mathcal{F}=\{\ph_a\}_{a\in \Omega_m}$ satisfies these assumptions, then we must have $|\Omega_m|\gg m$.
\end{rem}
One should note that the implicit upper bound for $\Phi_{\mathcal{F}}(V)$ in Theorem \ref{Main} holds in the slightly larger range $B'\le V\leq (N/\pi)(\log\log m-\log\log\log m)-B'$ for some constant $B'$ that depends at most on the parameters in the assumptions of Theorem \ref{Main}. 
Moreover, our proof of the implicit lower bound gives a much more precise estimate. In this case only Assumptions 2 and 3 are needed.  
\begin{thm}\label{Lower}
Let $m$ be large, and $\mathcal{F}=\{\ph_a\}_{a\in \Omega_m}$ be a family of $m$-periodic complex valued functions satisfying Assumptions 2 and 3 above. For all real numbers $1\le V\leq (N/\pi)(\log\log m-2\log\log\log m-B)$ we have 
$$ \Phi_{\mathcal{F}}(V) \geq \exp\left(-A_0 \exp\left(\frac{\pi}{N} V\right) \left(1+O\left(V e^{-\pi V/(2N)}\right)\right)\right)
$$  where 
\begin{equation}\label{A0}
A_0= \frac{N}{2}\exp\left(-\gamma-1 -\frac{1}{2N} \int_{-\infty}^{\infty} \frac{\f(u)}{u^2}du \right), \ \  B= \log A_0 +9,
\end{equation}
 $\gamma$ is the Euler-Mascheroni constant, and $\f:\mathbb{R}\to \mathbb{R}$ is defined by
\begin{equation}\label{Thefunctionf}
\f(t):=\begin{cases} \log \ex (e^{t\X}) & \text{ if } |t| <1,\\ \log \ex (e^{t\X})- N|t| & \text{ if } |t|\geq 1,\end{cases}
\end{equation}
where $\X$ is a random variable with the same distribution as the $\{\X(h)\}_{h\in \mathbb{Z}^*}$ in Assumption 3 above.
\end{thm}
As an application of Theorem \ref{Lower} (more specifically of Theorem \ref{LowerB} which is stronger), we exhibit large values of partial sums in families of periodic functions $\{\ph_a\}_{a\in \Omega_m}$ satisfying Assumptions 2 and 3. This was obtained by Lamzouri \cite{La} for Birch and Kloosterman sums, and independently by Bonolis \cite{Bo} for more general trace functions (though with a smaller constant). 
\begin{cor}\label{Main3} Let $m$ be large, and $\mathcal{F}=\{\ph_a\}_{a\in \Omega_m}$ be a family of $m$-periodic complex valued functions satisfying Assumptions 2 and 3 above. There exist at least $|\Omega_m|^{1-1/\log\log m}$ elements $a\in \Omega_m$ such that 
 \begin{equation}\label{OmegaResult}
 \left|\sum_{0\leq n\leq m/2} \ph_a(n)\right|\ge \left(\frac{N}{\pi}+o(1)\right) \sqrt{m}\log\log m.
 \end{equation}
\end{cor}
Given a family $\mathcal{F}=\{\ph_a\}_{a\in \Omega_m}$ of $m$-periodic complex valued functions satisfying the assumptions in Theorem \ref{Main}, a natural question to ask is which of the bounds \eqref{PV2} and \eqref{OmegaResult} is optimal (up to a constant). If we suppose that the estimate \eqref{MainEstimate}
is valid in the whole ``viable'' range, that is for $1\ll V< V_{\text{max}}:= \max_{a\in \Omega_m} \M(\varphi_a)/\sqrt{m}$, then we would have
$$ 
 \exp\left(-\exp\left(\frac{\pi}{N} V_{\text{max}}+O(1)\right)\right)= \Phi_{\mathcal{F}}(V_{\text{max}}-o(1))\geq \frac{1}{|\Omega_m|},
$$
and hence
$$ 
V_{\text{max}}\leq \frac{N}{\pi} \log\log |\Omega_m|+O(1).
$$
In particular, if $|\Omega_m|\ll m^B$ with an absolute constant $B>0$ (which is the case in all the families we consider), this simple
heuristic argument suggests that 
\begin{equation}\label{HeuristBound}
 \max_{a\in \Omega_m} \M(\ph_a) \ll  \sqrt{m} \log\log m,
\end{equation}
a bound similar to the one proved by Montgomery and Vaughan for character sums under the assumption of the Generalized Riemann Hypothesis. Surprisingly, we show that unlike this case (in which multiplicativity plays a central role), the above heuristic argument is false for certain families of $m$-periodic complex valued functions satisfying the assumptions in case A)  of Theorem \ref{Main} (namely Assumption 2, and Assumption 3 with $\eta>1$).  More precisely, we construct such a family $\mathcal{F}=\{\ph_a\}_{a\in \Omega_m}$ with $|\Omega_m|\asymp m^3$, for which the P\'olya-Vinogradov inequality \eqref{PV2} is sharp (up to the value of the implicit constant). This suggests the existence of a transition in the behavior of the distribution function $\Phi_{\mathcal{F}}(V)$ near the maximal values. It also confirms the common belief in analytic number theory that the P\'olya-Vinogradov inequality, though simple to derive, is extremely difficult to improve. 
\begin{pro}\label{HugeFamily}
Let $m$ be large. There exists a family $\mathcal{F}=\{\ph_a\}_{a\in \Omega_m}$ of $m$-periodic complex valued functions satisfying Assumption 2 with $N=1$ and Assumption 3 with $\eta=4/3$, such that $|\Omega_m|\asymp m^3$ and 
$$ \max_{a\in \Omega_m} \M(\ph_a) \geq \frac{1}{\pi}\sqrt{m}\log m +O(\sqrt{m}).$$
\end{pro}

\begin{rem}
The family we construct in Proposition \ref{HugeFamily} does not satisfy Assumption 1. In fact one has $\max_{a\in \Omega_m}||\ph_a||_{\infty}\asymp \sqrt{m}$ for this family. One therefore wonders whether a similar result to Proposition \ref{HugeFamily} holds for certain families of $m$-periodic complex valued functions satisfying the assumptions in case C) of Theorem \ref{Main}, which is the case of most interest. Unfortunately, we were unable to construct such families. However, it seems plausible that in this case there are less fluctuations in the partial sums of $\ph_a$, and that a bound similar to \eqref{HeuristBound} holds. \end{rem} 

It follows from the results of Kowalski and Sawin \cite{KoSa} that
$$ \lim_{p\to\infty} \Phi_{\textup{Kl}}(V)= \lim_{p\to\infty}\Phi_{\textup{Bi}}(V)=\pr(\mathbb{M}_{\textup{st}}>V),$$
for any fixed $V$ for which $\pr(\mathbb{M}_{\textup{st}}>V)$ is continuous, where 
\begin{equation}\label{RandomKoSa}
\mathbb{M}_{\textup{st}}=\max_{\alpha\in [0, 1)} \left| \alpha \Y(0)+ \sum_{h\neq 0} \frac{e(\alpha h)-1}{2\pi i h} \Y(h)\right|,
\end{equation}
and $\{\Y(h)\}_{h\in \mathbb{Z}}$ is a sequence of independent random variables with Sato-Tate distributions on $[-2, 2]$.
A straightforward generalization of their argument shows that if $m$ is large and $\mathcal{F}=\{\ph_a\}_{a\in \Omega_m}$ is a family of $m$-periodic complex valued functions satisfying Assumptions 1, 2, 3, and 4, then for $V\ge 1$ fixed we have
$$\lim_{m\to \infty} \Phi_{\mathcal F}(V)=\pr(\mathbb{M}_{\X}>V),$$
where 
$$
\mathbb{M}_{\X}=\max_{\alpha\in [0, 1)} \left| \alpha \X(0)+ \sum_{h\neq 0} \frac{e(\alpha h)-1}{2\pi i h} \X(h)\right|,
$$
and $\{\X(h)\}_{h\in \mathbb{Z}}$ is a sequence of I.I.D. random variables supported on $[-N, N]$ and satisfying Assumptions 3a and 3b above.
Combining this result with Theorem \ref{Main} leads to the following estimate for the large deviations of the random model $\mathbb{M}_{\X}$, which improves on the estimates of Lamzouri \cite{La} and Kowalski-Sawin \cite{KoSa} for the large deviations of $\mathbb{M}_{\textup{st}}$.
\begin{cor}\label{LargeDevRand}
Let $\{\X(h)\}_{h\in \mathbb{Z}}$ be a sequence of I.I.D. random variables supported on $[-N, N]$ and satisfying Assumptions 3a and 3b above. For all $V\gg 1$ we have 
$$ \pr(\mathbb{M}_{\X}>V)= \exp\left(- \exp\left(\frac{\pi}{N}V+O(1)\right)\right).$$

\end{cor}

\subsection{Examples of families of $\ell$-adic trace functions satisfying our assumptions}\label{sec : examp} We exhibit several examples of families of exponential sums that satisfy the assumptions in part C) of Theorem \ref{Main}, namely Assumptions 1, 2, 4, and Assumption 3 with $\eta=1/2$. These correspond to families of $\ell$-adic trace functions which satisfy several conditions, and notably that their arithmetic and geometric monodromy groups are both equal to $\Sp_{2r}(\mathbb{C})$, for a certain integer $r\ge 1$. We shall describe these families in details in section \ref{TraceSection}. In particular, we obtain the following applications of Theorem \ref{Main}. In all of these examples, $m=p$ is a large prime, and $\Omega_p= \Fs$ or $\Omega_p=\Fs\times\Fs.$ The first corollary concerns generalizations of Birch sums (\cite[$7.13$, $\Sp$-example $(2)$]{Kat90}).
\begin{cor}\label{GBirch} Let $g\in \mathbb{Z}[t]$ be an odd polynomial of degree $2r+1$, such that $r\ge 1$. Let $\mathcal{F}_1= \{\ph_a\}_{a\in \Fs}$ where $\ph_a(n)= e_p(an+g(n))$. There exists a constant $B_1$ such that for all real numbers $B_1\le V\leq (2r/\pi)(\log\log p-2\log\log\log p)-B_1$ we have 
$$  \Phi_{\mathcal{F}_1}(V) = \exp\left(- \exp\left(\frac{\pi}{2r}V+O(1)\right)\right).
$$
Here $B_1$ and the implied constant depend only on $r$.

\end{cor}

The next application concerns generalizations of the classical Kloosterman sums (\cite[$7.12.3.1$]{Kat90}). 

\begin{cor}\label{GKloo} 
Let $r\ge 1$ be an odd integer, and $\mathcal{F}_2= \{\ph_{(a,b)}\}_{(a, b)\in \Fs\times\Fs}$ where $\ph_{(a, b)}(n)= e_{p}(bn+(a\overline{n})^{r}).$
There exists a constant $B_2$ such that for all real numbers $B_2\leq V\leq ((r+1)/\pi)(\log\log p-2\log\log\log p)-B_2$ we have 
$$  \Phi_{\mathcal{F}_2}(V) = \exp\left(- \exp\left(\frac{\pi}{r+1}V+O(1)\right)\right).
$$

\end{cor}




Finally our last application concerns additive twists of hyper-Kloosterman sums. Recall that for an integer $r\ge 2$, the $r$-th hyper-Kloosterman sum on $\F$ is defined for $n\in \Fs$ by 
$$\Kl_{r}(n;p)=\frac{(-1)^{r-1}}{p^{(r-1)/2}}\sum_{\substack{y_{1},...,y_{r}\in\mathbb{F}_{p}^{\times}\\y_{1}\cdot...\cdot y_{r}=n}}e_p\left(y_{1}+\cdots +y_{r}\right).$$

\begin{cor}\label{HypKloo}
Let $r\ge 3$ be an odd integer, and  
$\mathcal{F}_4= \{\ph_{(a,b)}\}_{(a, b)\in \Fs\times\Fs}$ where $\ph_{a, b}(n)= \Kl_{r}(\overline{an};p)e_{p}(bn)$.
There exists a constant $B_3$ such that for all real numbers $B_3\leq V\leq ((r+1)/\pi)(\log\log p-2\log\log\log p)-B_3$ we have 
$$  \Phi_{\mathcal{F}_4}(V) = \exp\left(- \exp\left(\frac{\pi}{r+1}V+O(1)\right)\right).
$$

\end{cor}

Our method also works in the case where the Fourier transforms $\widehat{\ph_a}$ are complex valued, but yields weaker estimates for $\Phi_{\mathcal{F}}$ in this case. This corresponds for example to certain families of $\ell$-adic trace functions whose monodromy group is $\textup{SL}_N(\mathbb{C})$ for some integer $N\geq 3$ (since in the case $N=2$ we have $\textup{SL}_2(\mathbb{C})=\textup{Sp}_2(\mathbb{C})$). In this case we need to change Assumption 3 to include all mixed moments of $\widehat{\ph_a}(h_1), \dots, \widehat{\ph_a}(h_k)$ and their complex conjugates. We also assume that the $\{\X(h)\}_{h\in \mathbb{Z}^*}$ are supported inside the disk $\{z\in \mathbb{C}: |z|\leq N\}$, and replace  $\X$ by $\re\X$ in Assumption 3a.
Then, given a family $\mathcal{F}=\{\ph_a\}_{a\in \Omega_m}$ of $m$-periodic complex valued functions such that $|\ta(h)|\leq N$, and $\mathcal{F}$ verifies (the new) Assumption 3 with $\eta>1$, or Assumptions 1 and 3 with $1/2<\eta\le 1$, or Assumptions 1, 4, and 3 with $\eta=1/2$, we can prove that in the range  $B\le V\leq (N/\pi)(\log\log m-2\log\log\log m)-B$ we have  \footnote{In the case of families of trace functions whose arithmetic and geometric monodromy groups are both equal to $\textup{SL}_N(\mathbb{C})$ for some odd integer $N\ge 3$, the constant $\pi/N$ should be replaced by the larger constant $2\pi/(N(1+\cos(\pi/N)))$ in the lower bound for $\Phi_{\mathcal F}(V)$, since the condition $\pr(\re\X<-N+\varepsilon)\gg \varepsilon^A$ in Assumption 3a is not satisfied in this case.}
$$  \exp\left(- \exp\left(\frac{\pi}{N}V+O(1)\right)\right) \leq \Phi_{\mathcal{F}}(V) \leq \exp\left(-\exp\left(\frac{\pi^2}{4N}V+O(1)\right)\right).
$$

The plan of the paper is as follows. In the next section we present the key ingredients of the proof of the upper bound of Theorem \ref{Main}, and show how to deduce this upper bound in each of the cases A), B) and C) assuming these results. In section \ref{ShortSec} we prove Theorem \ref{ReduceMaxPoints} below, which shows that for almost all $m$-periodic functions in our families, the maximum of all partial sums is very close to the maximum of a ``small number" of these sums. Section \ref{AsympGSec} will be devoted to the proof of Theorem \ref{AsympG} below, which provides a precise asymptotic formula for the maximum of a certain ``random'' exponential sum.
In section \ref{ProbabilitySec}, we collect several results on the probabilistic random model.  Section \ref{MomentsSec} contains the proof of Theorem \ref{MomentsMaxTail} below, which is the last ingredient of the proof of the upper bound of Theorem \ref{Main}. In section \ref{LowerSec}, we prove Theorem \ref{Lower}. Section \ref{HugeSec} contains the construction of the family $\mathcal{F}$ which satisfies Proposition \ref{HugeFamily}. Finally, in section \ref{TraceSection}, we exhibit examples of families of $\ell$-adic trace functions satisfying our assumptions, and prove Corollaries \ref{GBirch}, \ref{GKloo} and \ref{HypKloo}.

\textbf{Acknowledgements}. 
We would like to thank the anonymous referees for carefully reading the paper and for their remarks and suggestions.


\section{Proof of the upper bound in Theorem \ref{Main}: Main ideas and key ingredients\label{sectionKey}}

Let $\{\ph_a\}_{a\in \Omega_m}$ be a family of $m$-periodic complex valued functions satisfying Assumptions 2 and 3. Recall that 
$$ \M(\ph_a) =\max_{0\leq x<m}\left|\sum_{0\leq n\le x}\ph_a(n)\right|.$$
Using the discrete Plancherel formula \eqref{Planche} and the estimate \eqref{Planche2} we obtain
$$
 \frac{\M(\ph_a)}{\sqrt{m}}= \frac{1}{2\pi} \max_{0\leq j\le m-1}\Big|\sum_{1\leq |h|<m/2} \frac{e_m\left(jh\right)-1}{ h} \ta(h)\Big|+O\left(1\right).
$$
Similarly as in \cite{La}, we shall treat the Fourier transforms $\ta(h)$ for small $h$ as \emph{random} values in $[-N, N]$. This yields 
 \begin{equation}\label{UpperReducePoints}
\frac{\M(\ph_a)}{\sqrt{m}}\leq  
\frac{N}{2\pi}\mathcal{G}(H) + \frac{1}{2\pi} \max_{0\leq j\le m-1} \left|\sum_{H<|h|<m/2} \frac{e_m\left(j h\right)- 1}{ h} \ta(h)\right| +O(1),
\end{equation}
where $H$ is a positive integer and
$$\mathcal{G}(H):=\max_{\alpha\in [0, 1)} \max_{(y_{-H}, \dots y_{-1}, y_1,\dots,  y_{H})\in [-1, 1]^{2H}}
 \left|\sum_{1\leq |h|\leq H} \frac{e(\alpha h)-1}{h} y_h\right|.
 $$
 One has the trivial bounds 
\begin{equation}\label{BoundsG}
2\log H +O(1) \leq \mathcal{G}(H) \leq 4\log H +O(1),
\end{equation}
where the upper bound follows from the trivial inequality $|e(\alpha h)-1|\leq 2$, and the lower bound follows by taking $\alpha=1/2$, $y_h=-1$ if
$h>0$ and $y_h=1$ if $h<0$. Using Fourier analytic techniques, the third author showed in \cite{La} that 
$$\G(H)\leq \left(1+\frac{4}{\pi}\right) \log H+O(1),$$
 and conjectured that the lower bound of \eqref{BoundsG} is closer to the true order of magnitude of $\G(H)$. In section \ref{AsympGSec} we shall prove a stronger form of this conjecture.
\begin{thm}\label{AsympG}
Let $H$ be a positive integer. Then, we have 
$$\G(H)=2\log H+2\log 2+2\gamma+O\left(\frac{1}{H}\right).$$
\end{thm}
In order to prove the upper bound in Theorem \ref{Main}, it remains to show that for large $H$, and for ``most'' $a\in \Omega_m$, the maximum of the sum $|\sum_{H< |h|< m/2} \frac{e_m\left(j h\right)- 1}{ h} \ta(h)|$ is ``small''. To this end we prove the following result in section \ref{MomentsSec}.

\begin{thm}\label{MomentsMaxTail}
Let $m$ be large, and $k$ be an integer such that $10^5N^2<k\leq (\log m)/(50\log\log m)$. Let $\{\ph_a\}_{a\in \Omega_m}$ be a family of $m$-periodic complex valued functions satisfying Assumptions 2 and 3. Let $S$ be a non-empty subset of $[0, 1)$, and put $y=10^5 N^2 k$. Then we have
$$ 
\frac{1}{|\Omega_m|}\sum_{a\in \Omega_m} \max_{\alpha\in S}\left|\sum_{y\le |h|<m/2} \frac{e\left(\alpha h\right)- 1}{h} \ta(h)\right|^{2k} \ll e^{-2k} + \frac{|S|(4C_1\log m)^{8k}}{m^{\eta}}.
$$
\end{thm}
In the case where $\{\ph_a\}_{a\in \Omega_m}$ satisfies Assumption 3 with $\eta>1$, we can deduce the upper bound of Theorem \ref{Main} from Theorems \ref{AsympG} and \ref{MomentsMaxTail}.
\begin{proof}[Proof of the upper bound in case A) of Theorem \ref{Main}]
Let $k \leq (\eta-1)\log m/(50\log\log m)$ be a large positive integer to be chosen, and put $H=10^5 N^2 k$. First, combining equation \eqref{UpperReducePoints} with Theorem \ref{AsympG} we deduce that 
$$ \frac{\M(\ph_a)}{\sqrt{m}} \leq \frac{N}{\pi}\log k+ \frac{1}{2\pi} \max_{0\leq j\le m-1} \left|\sum_{H\leq |h|<m/2} \frac{e_m\left(j h\right)- 1}{ h} \ta(h)\right| +C_0,$$
for some positive constant $C_0$. We assume that $V$ is sufficiently large and choose $k=[C_2\exp(\pi V/N)]$, where $C_2= \exp(-\frac{\pi}{N}(C_0+\frac{1}{2\pi}))$. Therefore, appealing to Theorem \ref{MomentsMaxTail} with $S=\{j/m : 0\leq j\le m-1\}$ we obtain
\begin{equation}\label{DistribuMoments}
\begin{aligned}
\Phi_{\mathcal{F}}(V) 
&\leq \frac{1}{|\Omega_m|} \left|\left\{a\in \Omega_m: \max_{0\leq j\le m-1} \left|\sum_{H\leq |h|<m/2} \frac{e_m\left(j h\right)- 1}{ h} \ta(h)\right| \geq 1\right\}\right| \\
& \leq \frac{1}{|\Omega_m|} \sum_{a\in \Omega_m} \max_{0\leq j\le m-1} \left|\sum_{H\leq |h|<m/2} \frac{e_m\left(j h\right)- 1}{ h} \ta(h)\right|^{2k} \\
& \ll e^{-2k} + (4\log m)^{10k} m^{1-\eta} \ll \exp\left(-C_2\exp\left(\frac{\pi}{N}V\right)\right),
\end{aligned}
\end{equation}
as desired.

\end{proof}
If $\{\ph_a\}_{a\in \Omega_m}$ satisfies Assumption 3 with $\eta\leq 1$ (which corresponds to cases B) and C) of Theorem \ref{Main}), then the above argument no longer works since $|\{j/m : 0\leq j\le m-1\}|=m$ is too big. To overcome this problem, we shall suppose that our family satisfies Assumption 1, and use it to reduce the number of points $j\le m-1$ where the maximum of $\left|\sum_{0\leq n\le j}\ph_a(n)\right|$ can occur. Let $J\leq \sqrt{m}$ be a parameter to be chosen, and split the interval $[0,m]$ into $J$ intervals $I_{j}:=[x_{j},x_{j+1}]$ where for each $j=0,..., J$ we put
$$
x_{j}:=\frac{j}{J}m.
$$
We first consider case B) of Theorem \ref{Main} (where $\eta \in (1/2, 1]$) since it is easier. 

\begin{proof}[Proof of the upper bound in case B) of Theorem \ref{Main}]
We choose $J=\lfloor\sqrt{m}\rfloor$. For $a\in \Omega_m$, let $r_{a}$ be an integer in the interval $[0,m)$ such that
\[
\mathcal{M}(\ph_a)= \Big|\sum_{0\leq n\leq r_a}\ph_a(n)\Big|.
\]
Then there exists $0\leq j\leq J-1$ such that $r_a\in [x_{j},x_{j+1}]$, and hence
\begin{align*}
\Big|\frac{1}{\sqrt{m}}\sum_{0\leq n\leq r_a}\ph_a(n)\Big|
& \leq \Big|\frac{1}{\sqrt{m}}\sum_{0\leq n\leq x_{j}}\ph_a(n)\Big|+\Big|\frac{1}{\sqrt{m}}\sum_{ x_{j}< n\leq r_a}\ph_a(n)\Big|\\
& \leq \Big|\frac{1}{\sqrt{m}}\sum_{0\leq n\leq x_{j}}\ph_a(n)\Big|+ O(1),
\end{align*}
since $\max_{a\in \Omega_m}||\ph_a||\ll 1$ and $|r_a-x_j|\leq m/J\ll \sqrt{m}.$ This implies that 
$$  \frac{\M(\ph_a)}{\sqrt{m}}= \max_{0\leq j\leq J-1} \Big|\frac{1}{\sqrt{m}}\sum_{0\leq n\leq x_{j}}\ph_a(n)\Big| +O(1).$$ 
We now use the same argument leading up to \eqref{DistribuMoments} with the same choices of $k \leq (\eta-1/2)\log m/(30\log\log m)$ and $H$, but with $S=\{x_j/m : 0\leq j\leq J-1\}$ (and perhaps a different choice for the constant $C_0$).  This gives 
\begin{align*}
\Phi_{\mathcal{F}}(V) 
& \leq \frac{1}{|\Omega_m|} \sum_{a\in \Omega_m} \max_{0\leq j\leq J-1} \left|\sum_{H\leq |h|<m/2} \frac{e_m\left(x_j h\right)- 1}{ h} \ta(h)\right|^{2k} \\
& \ll e^{-2k} + (4\log m)^{10k} m^{1/2-\eta} \ll \exp\left(-C_2\exp\left(\frac{\pi}{N}V\right)\right).
\end{align*}
\end{proof}
The above argument fails if $\{\ph_a\}_{a\in \Omega_m}$ satisfies Assumption 3 with $\eta=1/2$, which is the most interesting case of Theorem \ref{Main}. In this case, to reduce the number of points of $S$ further (below $m^{1/2-\varepsilon}$ for some $\varepsilon$),  we need power saving bounds for short sums $\sum_{x\leq n\leq x+h} \ph_a(n)$ in \emph{the P\'olya-Vinogradov range}, which corresponds to $h$ being of size around $\sqrt{m}$.
Unfortunately, such bounds are only known in very few cases (for example they are known for Birch sums but not for Kloosterman sums). To overcome this problem, we use Assumption 4 in order to obtain strong bounds for these short sums uniformly over all intervals $I$ of length $|I|\leq m^{1/2+\delta/2}$ (intervals at the edge of the P\'olya-Vinogradov range), in an average sense. In fact, it is this uniformity aspect (see Lemma \ref{BoundMomentShort} below) that allows us to obtain the upper bound of Theorem \ref{Main} in this case.
Let $\alpha$ and $\delta$ be as in Assumption 4. As before we will split the interval $[0,m]$ into $J$ intervals $I_{j}:=[x_{j},x_{j+1}]$ where 
$x_{j}:=\frac{j}{J}m$, and where we now choose $J= \lfloor m^{1/2-\delta/5}\rfloor$. We shall prove the following result in section \ref{ShortSec}.
\begin{thm}\label{ReduceMaxPoints}
Let $m$ be large and $J= \lfloor m^{1/2-\delta/5}\rfloor$. Let $\{\ph_a\}_{a\in \Omega_m}$ be a family of $m$-periodic complex valued functions satisfying Assumptions 1 and 4. 
There exists a set $\E_m\subset \Omega_m$ with $|\E_m|\leq m^{-\delta/10} |\Omega_m|$ such that for all $a\in \Omega_m\setminus \E_m$ we have 
$$ \M(\ph_a)= \max_{0\leq j\leq J-1}\Big|\sum_{0\leq n\leq x_{j}}\ph_a(n)\Big|+O\left(m^{1/2-\delta/(8\alpha)}\right).$$
\end{thm}
We end this section by deducing the upper bound in case C) of Theorem \ref{Main} from Theorems \ref{AsympG}, \ref{MomentsMaxTail} and \ref{ReduceMaxPoints}.
\begin{proof}[Proof of the upper bound in case C) of Theorem \ref{Main}]
Let $\E_m$ be the exceptional set in Theorem \ref{ReduceMaxPoints}, and $a\in \Omega_m\setminus \E_m$. 
Combining this result with the discrete Plancherel formula \eqref{Planche} and the estimate \eqref{Planche2} we obtain
$$
 \frac{\M(\ph_a)}{\sqrt{m}}= \frac{1}{2\pi} \max_{0\leq j\leq J-1}\Big|\sum_{1\leq |h|<m/2} \frac{e_m\left(x_j h\right)-1}{ h} \ta(h)\Big|+O\left(1\right).
$$
Let $k \leq {\delta}(\log m)/(200\log\log m)$ be a large positive integer to be chosen, and put $H=10^5 N^2 k$. First, combining equation \eqref{UpperReducePoints} with Theorem \ref{AsympG}, we deduce that if $a\in \Omega_m\setminus \mathcal{E}_m$ we have 
$$ \frac{\M(\ph_a)}{\sqrt{m}} \leq \frac{N}{\pi}\log k+ \frac{1}{2\pi} \max_{0\leq j\leq J-1} \left|\sum_{H\leq |h|<m/2} \frac{e_m\left(x_j h\right)- 1}{ h} \ta(h)\right| +C_0,$$
for some positive constant $C_0$. Repeating the same argument leading up to \eqref{DistribuMoments} with the same choice of $k$ gives
\begin{align*}
\Phi_{\mathcal{F}}(V) 
&\leq \frac{1}{|\Omega_m|} \left|\left\{a\in \Omega_m\setminus \mathcal{E}_m: \max_{0\leq j\leq J-1} \left|\sum_{H\leq |h|<m/2} \frac{e_m\left(x_j h\right)- 1}{ h} \ta(h)\right| \geq 1\right\}\right| +O\left(\frac{|\mathcal{E}_m|}{|\Omega_m|}\right)\\
& \leq \frac{1}{|\Omega_m|} \sum_{a\in \Omega_m} \max_{0\leq j\leq J-1} \left|\sum_{H\leq |h|<m/2} \frac{e_m\left(x_j h\right)- 1}{ h} \ta(h)\right|^{2k} +O\left(m^{-\delta/10}\right)\\
& \ll e^{-2k} + (4\log m)^{10k} m^{-\delta/10} \ll \exp\left(-C_2\exp\left(\frac{\pi}{N}V\right)\right).
\end{align*}
\end{proof}


\section{Controlling short sums of periodic functions: Proof of Theorem \ref{ReduceMaxPoints} \label{ShortSec}}

In order to prove Theorem \ref{ReduceMaxPoints}, we will use Assumptions 1 and 4 to obtain a non-trivial upper bound for the $\alpha$-th moment of the \emph{maximum} over intervals $I$ (with length up to a certain parameter $L$) of the short sum $\sum_{n\in I} \ph_a(n)$. 
\begin{lem}\label{BoundMomentShort}
Let $m$ be large, and $\{\ph_a\}_{a\in \Omega_m}$ be a family of $m$-periodic complex valued functions satisfying Assumptions 1 and 4. For any real number $1\leq L\leq m^{1/2+\delta/2}$ we have
\[
\frac{1}{|\Omega_{m}|}\sum_{a\in\Omega_m}\max_{|I|\leq L}\Big|\frac{1}{\sqrt{m}}\sum_{n\in I}\ph_a(n)\Big|^{\alpha}\ll Lm^{-1/2-\delta/2}+m^{-\delta/4},
\]
where the maximum is taken over all intervals $I=[x, y]\subset [0, m]$ with $|I|\leq L$. 
\end{lem}

\begin{proof}
The intervals $I=[x,y]$ with  $0\leq x<y\leq m$, and $|I|=y-x\leq L$ can be parametrized by the set of points in the region of the plane delimited by the trapezoid  
\[
T_{L}:=\{(x,y)\in\mathbb{R}^{2}: 0\leq x<y\leq m,\quad y\leq x+L\}.
\]
Let $0<B\leq \sqrt{m}$ be a parameter to be chosen, and for any $k,\ell\in\mathbb{N}$ we define $S_{k,\ell,B}:=[kB,(k+1)B)\times [\ell B,(\ell+1)B)$. The set of squares given by
\[
\{S_{k,\ell ,B} \ | \ S_{k,\ell,B}\cap T_{L}\neq\emptyset\}
\]
is a disjoint cover of $T_{L}$ of size
\begin{equation}\label{NumberSquares}
\begin{split}
\mathcal{N}_{L,B}:=|\{S_{k,\ell,B} \ | \ S_{k,\ell,B}\cap T_{L}\neq\emptyset\}|&\ll\frac{A(T_{L})}{B^{2}} + \frac{m}{B}
\\&\ll\frac{mL}{B^{2}}+\frac{m}{B},
\end{split}
\end{equation}
where $A(D)$ denotes the area of $D$. 
For any $a\in\Omega_{m}$ let us denote by $I_{a}=[x_{a},y_{a}]$ an interval with $|I_{a}|\leq L$ such that
\[
\Big|\frac{1}{\sqrt{m}}\sum_{ n\in I_{a}}\ph_{a}(n)\Big|=\max_{|I|\leq L}\Big|\frac{1}{\sqrt{m}}\sum_{ n\in I}\ph_{a}(n)\Big|.
\]
Then there exists $k_{a},\ell_{a}\in\mathbb{N}$ such that $(x_{a},y_{a})\in S_{k_a,\ell_a,B}$. Hence
\[
\begin{split}
\frac{1}{\sqrt{m}}\sum_{ n\in I_a}\ph_a(n)&=\frac{1}{\sqrt{m}}\sum_{k_aB\leq n\leq \ell_{a}B}\ph_a(n)+O\Big(\frac{x_{a}-k_{a}B+y_{a}-\ell_{a}B}{\sqrt{m}}\Big)\\&=\frac{1}{\sqrt{m}}\sum_{k_{a}B\leq n\leq \ell_{a}B}\ph_{a}(n)+O\Big(\frac{B}{\sqrt{m}}\Big),
\end{split}
\]
by Assumption 1.
Using this estimate together with the elementary inequality $|x+y|^{\alpha}\leq (2 \max(|x|, |y|))^{\alpha}\leq 2^{\alpha}(|x|^{\alpha}+|y|^{\alpha})$ we get
\[
\begin{split}
\frac{1}{|\Omega_{m}|}\sum_{a\in\Omega_{m}}\max_{|I|\leq L}\Big|\frac{1}{\sqrt{m}}\sum_{n\in I}\ph_{a}(n)\Big|^{\alpha}&=\frac{1}{|\Omega_m|}\sum_{a\in\Omega_m}\Big|\frac{1}{\sqrt{m}}\sum_{k_a B\leq n\leq \ell_{a} B}\ph_a(n)+ O\Big(\frac{B}{\sqrt{m}}\Big)\Big|^{\alpha}\\&\ll \frac{1}{|\Omega_m|}\sum_{a\in\Omega_m}\Big|\frac{1}{\sqrt{m}}\sum_{k_a B \leq n\leq \ell_a B}\ph_a(n)\Big|^{\alpha}+\frac{B^{\alpha}}{m^{\alpha/2}}.
\end{split}
\]
Furthermore, observe that
\[
\begin{split}
\frac{1}{|\Omega_m|}\sum_{a\in\Omega_m}\Big|\frac{1}{\sqrt{m}}\sum_{k_a B\leq n\leq \ell_a B}\ph_a(n)\Big|^{\alpha}&\leq \frac{1}{|\Omega_m|}\sum_{a\in\Omega_m}\sum_{\substack{k,\ell\in\mathbb{N} :\\ S_{k,\ell, B}\cap T_{L}\neq\emptyset}}\Big|\frac{1}{\sqrt{m}}\sum_{kB\leq n\leq \ell B}\ph_a(n)\Big|^{\alpha}\\&= \sum_{\substack{k,\ell\in\mathbb{N} :\\ S_{k,\ell,B}\cap T_{L}\neq\emptyset}}\frac{1}{|\Omega_m|}\sum_{a\in\Omega_m}\Big|\frac{1}{\sqrt{m}}\sum_{kB\leq n\leq \ell B}\ph_a(n)\Big|^{\alpha}\\
& \ll m^{-1/2-\delta} \mathcal{N}_{L, B}.
\end{split}
\]
by Assumption 4, since $B\leq m^{1/2}$. 
Therefore, we deduce from \eqref{NumberSquares} that
\[
\frac{1}{|\Omega_m|}\sum_{a\in\Omega_m}\max_{|I|\leq L}\Big|\frac{1}{\sqrt{m}}\sum_{n\in I}\ph_a(n)\Big|^{\alpha}\ll \frac{L m^{1/2-\delta}}{B^{2}}+\frac{m^{1/2-\delta}}{B}+\frac{B^{\alpha}}{m^{\alpha/2}}.
\]
Choosing $B=m^{1/2-\delta/4}$ gives the result. 
\end{proof}

\begin{proof}[Proof of Theorem \ref{ReduceMaxPoints}]
It only suffices to prove the implicit upper bound, since one trivially has 
$$ \M(\ph_a)\geq \max_{1\leq j\leq J-1}\Big|\sum_{0\le n\le x_{j}}\ph_a(n)\Big|.$$
Let $L=m^{1/2+\delta/4}$ and define $\E_m$ to be the set of elements $a\in \Omega_m$ such that 
$$\max_{|I|\leq L}\Big|\frac{1}{\sqrt{m}}\sum_{n\in I}\ph_a(n)\Big| > m^{-\delta/(8\alpha)} .$$
Then, it follows from Lemma \ref{BoundMomentShort} that 
$$
\frac{|\E_m|}{|\Omega_m|}\leq
\frac{m^{\frac{\delta}{8}}}{|\Omega_m|}\sum_{a \in \Omega_m}\max_{|I|\le L}\Big|\frac{1}{\sqrt{m}}\sum_{n\in I}\ph_a(n)\Big|^{\alpha}\ll m^{-\delta/8}.
$$
Moreover, for $a\in \Omega_m$, let $r_{a}$ be an integer in the interval $[0,m)$ such that
\[
\mathcal{M}(\ph_a)= \Big|\sum_{0\leq n\leq r_a}\ph_a(n)\Big|.
\]
Then there exists $0\leq j\leq J-1$ such that $r_a\in [x_{j},x_{j+1}]$, and hence
$$
\Big|\frac{1}{\sqrt{m}}\sum_{0\leq n\leq r_a}\ph_a(n)\Big|\leq \Big|\frac{1}{\sqrt{m}}\sum_{0\leq n\leq x_{j}}\ph_a(n)\Big|+\Big|\frac{1}{\sqrt{m}}\sum_{ x_{j}< n\leq r_a}\ph_a(n)\Big|.
$$
Recall that $J= \lfloor m^{1/2-\delta/5}\rfloor$, and hence $r_a-x_j\leq x_{j+1}-x_j= m/J \leq L$ if $m$ is large enough. Therefore, we deduce that if $a\in \Omega_m\setminus \E_m$ then
\[
\Big|\frac{1}{\sqrt{m}}\sum_{ x_{j}< n\leq r_a}\ph_a(n)\Big|\le m^{-\delta/(8\alpha)}.
\]
This implies
$$ 
\frac{\M(\ph_a)}{\sqrt{m}}\leq \max_{1\leq j\leq J-1}\Big|\frac{1}{\sqrt{m}}\sum_{0\leq n\leq  x_{j}}\ph_a(n)\Big|+O\left(m^{-\delta/(8\alpha)}\right),
$$
for all $a\in \Omega_m\setminus \E_m$,  completing the proof. 
\end{proof}


\section{An asymptotic estimate for the maximum of a random sum: Proof of Theorem \ref{AsympG} \label{AsympGSec}}

Recall that
\[\G(H)=\max_{\alpha\in [0,1]} \max_{(x_h)_{1\le|h|\le H}\in[-1,1]^{2H}}\left|\sum_{1\le|h|\le H}\frac{e(\alpha h)-1}{h}x_h\right|.\]
We shall deduce Theorem \ref{AsympG} from the following result, which is an exact formula for $\G(H)$ when $H$ is odd. 
\begin{pro}\label{FormulaOddG}
If $H$ is an odd positive integer, then
\[\G(H)=2\sum_{h=1}^H\frac{1-(-1)^h}{h}.\]
\end{pro}
To prove this result, we need the following lemma. 
\begin{lem}\label{LemmaUpperSin}
Let $\alpha$ be a real number. If $H\geq 1$ is odd, then
\[\sum_{h=1}^H\frac{\sin^2(\pi\alpha h)}{h}\le\sum_{h=1}^H\frac{1-(-1)^h}{2h}.\]
\end{lem}

\begin{proof} Since $\sin^2(\pi(1-\alpha) h)=\sin^2(\pi\alpha h)$ we may assume that $\alpha\in[0,1/2]$.  Let $g:[0,1/2]\rightarrow\mathbb{R}$ be defined by 
\[g(t):=\frac{\cos(2\pi tH)}{4H}+\sum_{h=1}^H\frac{\sin^2(\pi th)}{h}= \frac{\cos(2\pi tH)}{4H} + \sum_{h=1}^H\frac{1-\cos(2\pi th)}{2h}.\]
Since $g$ is differentiable and
\begin{align*}
g'(t)&=-\frac{\pi}{2}\sin(2\pi tH)+\pi\sum_{h=1}^H\sin(2\pi th)\\
&= \pi \left(-\sin(\pi tH)\cos(\pi tH) + \frac{\sin(\pi tH)\sin(\pi t(H+1))}{\sin(\pi t)}\right) =\pi\frac{\sin^2(\pi tH)\cos(\pi t)}{\sin(\pi t)}\ge 0,
\end{align*}
we deduce that $g$ is increasing on $[0, 1/2]$. This implies that for all $\alpha\in [0, 1/2]$ we have 
$$ \sum_{h=1}^H\frac{\sin^2(\pi\alpha h)}{h}  \leq g(1/2) - \frac{\cos(2\pi \alpha H)}{4H} = \sum_{h=1}^H\frac{1-(-1)^h}{2h} - \frac{\cos(2\pi \alpha H)+1}{4H}\leq \sum_{h=1}^H\frac{1-(-1)^h}{2h},$$
as desired. 
\end{proof}

\begin{proof}[Proof of Proposition \ref{FormulaOddG}] The lower bound follows easily by taking $\alpha=1/2$, $x_h=-1$ if
$h>0$ and $x_h=1$ if $h<0$. Let us now show the upper bound. Let
$\alpha\in\mathbb{R}$ and $(x_h)_{1\le|h|\le H}\in[-1,1]^{2H}$. Put
\[S=\sum_{1\le|h|\le H}\frac{e(\alpha h)-1}{h}x_h\quad\textrm{and}\quad S_0=\sum_{h=1}^H\frac{\sin^2(\pi\alpha h)}{h}.\]
On one hand, we have the relation
\[\im (S)=2\sum_{h=1}^H\sin(\pi\alpha h)\cos(\pi\alpha h)\frac{x_h+x_{-h}}{h}.\]
Moreover, using the Cauchy-Schwarz inequality we get
\[\im (S)^2\le4\Bigl(\sum_{h=1}^H\frac{\sin^2(\pi\alpha h)}{h}|x_h+x_{-h}|\Bigr)\Bigl(\sum_{h=1}^H\frac{\cos^2(\pi\alpha h)}{h}|x_h+x_{-h}|\Bigr).\]
On the other hand, we have
\[\re (S)=2\sum_{h=1}^H\frac{\sin^2(\pi\alpha h)}{h}(x_{-h}-x_h).\]
Observe that $|x_h-x_{-h}|+|x_h+x_{-h}|\le2$. This implies the upper bound
\[|\re (S)|\le4S_0-2\sum_{h=1}^H\frac{\sin^2(\pi\alpha h)}{h}|x_h+x_{-h}|.\]
We are now ready to estimate $|S|^2=\re (S)^2+\im (S)^2$. We infer
\begin{align*}
|S|^2&\le16S_0^2-16S_0\sum_{h=1}^H\frac{\sin^2(\pi\alpha h)}{h}|x_h+x_{-h}|+4\Bigl(\sum_{h=1}^H\frac{\sin^2(\pi\alpha h)}{h}|x_h+x_{-h}|\Bigr)\sum_{h=1}^H\frac{|x_h+x_{-h}|}{h}\\
&\le16S_0^2+8\Bigl(\sum_{h=1}^H\frac{\sin^2(\pi\alpha h)}{h}|x_h+x_{-h}|\Bigr)\Bigl(\sum_{h=1}^H\frac{1}{h}-2S_0\Bigr).\\
\end{align*}
To finish the proof, let us study two cases:
\begin{itemize}
\item[i)] If $\displaystyle2S_0\ge\sum_{h=1}^H\frac{1}{h}$, then $|S|^2\le16S_0^2$
and we conclude by applying Lemma \ref{LemmaUpperSin}.
\item[ii)] If $\displaystyle2S_0\le\sum_{h=1}^H\frac{1}{h}$, then
\[|S|^2\le16S_0^2+16S_0\Bigl(\sum_{h=1}^H\frac{1}{h}-2S_0\Bigr)\le4\Bigl(\sum_{h=1}^H\frac{1}{h}\Bigr)^2\le\Bigl(2\sum_{h=1}^H\frac{1-(-1)^h}{h}\Bigr)^2.\]
\end{itemize}
Whence the result. 
\end{proof}
We end this section by deducing Theorem \ref{AsympG}. 
\begin{proof}[Proof of Theorem \ref{AsympG}]
If $H\geq 1$ is odd, the desired asymptotic follows directly from Proposition \ref{FormulaOddG}, so it only remains to prove the result when $H$ is even. To this end, we observe that if $k\geq 1$ is an integer, then we have
$$ \G(k+1)= \max_{\alpha\in [0,1]} \max_{(x_h)_{1\le|h|\le k+1}\in[-1,1]^{2k+2}}\left|\sum_{1\le|h|\le k+1}\frac{e(\alpha h)-1}{h}x_h\right| \geq \G(k), $$
which follows by taking $x_{k+1}=x_{-k-1}=0$. Hence it follows from Proposition \ref{FormulaOddG} applied to $\G(H-1)$ and $\G(H+1)$, together with the inequality $\G(H-1)\leq \G(H)\leq \G(H+1)$ that 
$$ \G(H)= 2\sum_{h=1}^H\frac{1-(-1)^h}{h} + O\left(\frac{1}{H}\right)=2\log H+2\log 2+2\gamma+O\left(\frac{1}{H}\right),
$$
as desired.
\end{proof}

\begin{rem}
Using the same method of proof as in Proposition \ref{FormulaOddG}, we can in fact obtain the following exact formula for $\G(H)$ when $H$ is even
\[\G(H)=2\sum_{h=1}^H\frac{1}{h}\Bigl(1-(-1)^h\cos\frac{\pi h}{H+1}\Bigr).\]
\end{rem}


\section{Investigating the probabilistic random model\label{ProbabilitySec}}

Let $\{\X(h)\}_{h\in \mathbb{Z}^*} $ be a sequence of independent random variables supported on $[-N, N]$ and satisfying Assumptions 3a and 3b above. In this section we shall study the moments and the moment generating function of the sum of random variables $\sum_{y\le |h|<z} c(h) \X(h)$, where $c(h)$ are certain complex numbers such that $c(h)\ll1/|h|$ for $|h|\geq 1$. 


It follows from the results of section \ref{TraceSection} below that the families of $\ell$-adic trace functions we consider in Corollaries \ref{GBirch}, \ref{GKloo} and \ref{HypKloo} satisfy Assumption 3, where the random variables $\X(h)$ are distributed like the traces of random matrices on the compact classical group $\mathrm{USp}_{2r}$. At the end of this section we will show that these random variables satisfy Assumptions 3a and 3b.

\subsection{The moments of $\sum_{y\le |h|<z} c(h) \X(h)$}
The purpose of this section is to prove the following lemma, which is a generalization of Lemma 3.1 of \cite{La}.
\begin{lem}\label{BOUNDRAND}
Let $\X(h)$ be a sequence of I.I.D. random variables satisfying Assumption 3b above. Let $\{c(h)\}_{h\in \mathbb{Z}^*}$ be a sequence of complex numbers such that $|c(h)|\leq c_0/|h|$, where $c_0$ is a positive constant. Let $1\leq y<z$ be real numbers. Then, for all integers $k\geq 1$ we have 
\begin{equation}\label{LargeY}
\ex\left(\left|\sum_{y\le |h|<z} c(h) \X(h)\right|^k\right) \leq \left(\frac{8(c_0N)^2k}{y}\right)^{k/2}.
\end{equation}
Moreover, if $k>y$ then 
\begin{equation}\label{SmallY}
\ex\left(\left|\sum_{y\le |h|<z} c(h) \X(h)\right|^k\right) \leq (10 c_0N \log k)^k.
\end{equation}
\end{lem}

\begin{proof}
We first prove \eqref{LargeY} when $k=2n$ is even. 
Expanding the moments we obtain
\begin{equation}\label{b1}
\ex\left(\left|\sum_{y\le |h|<z} c(h) \X(h)\right|^{2n}\right)\leq 
c_0^{2n}\sum_{y\le |h_1|, \dots, |h_{2n}|<z} \frac{\left|\ex\left(\X(h_1) \cdots \X(h_{2n})\right)\right|}{|h_{1}\cdots h_{2n}|}. 
\end{equation}
By the independence of the $\{\X(h)\}$ and Assumption 3b, we get
\begin{equation}\label{b2}
\begin{aligned}
&\sum_{y\le |h_1|, \dots, |h_{2n}|<z} \frac{\left|\ex\left(\X(h_1) \cdots \X(h_{2n})\right)\right|}{|h_{1}\cdots h_{2n}|}\\
 &= \sum_{\ell=1}^{2n}\sum_{\substack{j_1<\cdots <j_{\ell}\\ y\le |j_1|, \dots, |j_{\ell}|<z}}\sum_{\substack{n_1, \dots, n_{\ell}\geq 1\\ n_1+\cdots +n_{\ell}=2n}} \binom{2n}{n_1, \dots, n_{\ell}} \frac{|\ex(\X(j_1)^{n_1})| \cdots |\ex(\X(j_{\ell})^{n_{\ell}})|}{|j_{1}^{n_1}\cdots j_{\ell}^{n_{\ell}}|}\\
&\leq  N^{2n}\sum_{\ell=1}^{n}\sum_{\substack{j_1<\cdots <j_{\ell}\\ y\le |j_1|, \dots, |j_{\ell}|<z}}\sum_{\substack{r_1, \dots, r_{\ell}\geq 1\\ r_1+\cdots +r_{\ell}=n}} \binom{2n}{2r_1, \dots, 2r_{\ell}} \frac{1}{j_{1}^{2r_1}\cdots j_{\ell}^{2r_{\ell}}},\\
\end{aligned}
\end{equation}
where we have used that $\ex(\X(h)^m)=0$ if $m$ is odd, and $|\ex(\X(h)^m)|\leq N^m$ if $m$ is even, since $|\X(h)|\leq N$ for all $h$. 
Furthermore, observe that
$$ \binom{2n}{2r_1, \dots, 2r_{\ell}} \leq \frac{2n!}{n!}\binom{n}{r_1, \dots, r_{\ell}}\leq (2n)^n\binom{n}{r_1, \dots, r_{\ell}}.$$
 Inserting this bound in \eqref{b2} gives 
$$ \sum_{y\le |h_1|, \dots, |h_{2n}|<z} \frac{\left|\ex\left(\X(h_1) \cdots \X(h_{2n})\right)\right|}{|h_{1}\cdots h_{2n}|} 
\leq (2N^2n)^n \left(\sum_{y\leq |j|< z} \frac{1}{j^2}\right)^n.
$$
Therefore, in view of \eqref{b1} and the elementary inequality $\sum_{y\leq |j|< z} 1/j^2\leq 4/y$ we deduce that
\begin{equation}\label{EVENBound}
\ex\left(\left|\sum_{y\le |h|<z} c(h) \X(h)\right|^{2n}\right)\leq \left(\frac{8(c_0N)^2n}{y}\right)^{n}.
\end{equation}
We now establish \eqref{LargeY} when $k$ is odd. By the Cauchy-Schwarz inequality and \eqref{EVENBound} we have 
$$ \ex\left(\left|\sum_{y\le |h|<z} c(h) \X(h)\right|^{k}\right)\leq \ex\left(\left|\sum_{y\le |h|<z} c(h) \X(h)\right|^{2k}\right)^{1/2}\leq \left(\frac{8(c_0N)^2k}{y}\right)^{k/2},$$
as desired. 

We now prove \eqref{SmallY}. By \eqref{LargeY} and Minkowski's inequality we have 
\begin{align*}
\ex\left(\left|\sum_{y\le |h|<z} c(h) \X(h)\right|^k\right)^{1/k}
&\leq \ex\left(\left|\sum_{y\le |h| < k} c(h) \X(h)\right|^k\right)^{1/k} +\ex\left(\left|\sum_{k\le |h|<z} c(h) \X(h)\right|^k\right)^{1/k}\\
&\leq c_0 N\sum_{y\le |h|<k} \frac{1}{|h|}+ \sqrt{8}c_0N \leq 10 c_0 N \log k .\\
\end{align*}
This completes the proof.
\end{proof}

\subsection{The moment generating function of a sum involving the $\X(h)$}

In this section we shall estimate the moment generating function of the sum of random variables $\sum_{\substack{-m/2<h\le m/2\\ h\neq 0}} \gamma_m(h) \X(h)$, 
where the $\gamma_{m}(h)$ are defined by 
$$  \gamma_{m}(h)=:-\frac{1}{m} \im \sum_{0\leq n\leq m/2}e_m\left(nh\right).$$ This is in fact
 the probabilistic random model corresponding to the imaginary part of the partial sum $\sum_{0\leq n\leq x} \ph_a(n)$ when $x=m/2$. Indeed,  by \eqref{Planche}  we have
\begin{equation}\label{EqualityB}
\frac{1}{\sqrt{m}}\im \sum_{0\leq n\leq m/2} \ph_a(n)=\sum_{\substack{-m/2<h\le m/2\\ h\neq 0}}\gamma_m(h) \ta(h),
\end{equation}
since $\gamma_m(0)=0$. For $s\in \mathbb{C}$ we define
$$\mathcal{L}_{\X}(s):=\ex\Bigg(\exp\Bigg(s \cdot \sum_{\substack{-m/2<h\le m/2\\ h\neq 0}}\gamma_m(h) \X(h)\Bigg)\Bigg).$$
We prove the following proposition, which generalizes Proposition 3.2 of \cite{La}, and will be used to prove the lower bound of Theorem \ref{Main} in section \ref{LowerSec}.
\begin{pro}\label{AsympLaplace}
Let $m$ be a large integer and $2\leq s\leq m^{1/3}$ be a real number. Then we have
$$
\mathcal{L}_{\X}(s)=\exp\left(\frac{N}{\pi}s\log s+ B_0 s +O\left(\log^2 s\right)\right),
$$
where 
$$
B_0= \frac{N}{\pi}\left( \gamma+\log 2-\log \pi+ \frac{1}{2N} \int_{-\infty}^{\infty} \frac{\f(u)}{u^2}du\right).
$$
\end{pro}

To prove this result we need the following lemma.

\begin{lem}\label{LogEx}
Let $\X$ be a random variable with values in $[-N, N]$, such that $\ex(\X)=0$ and $\X$ satisfies Assumption 3a. Let $f_{\X}$ be the function defined in \eqref{Thefunctionf}.  Then we have the following estimates
\begin{equation}\label{LogEx1}
\f(t)\ll \begin{cases} t^2 & \text{ if }  |t|<1,\\ \log(2 |t|)  & \text{ if } |t|\geq 1,\end{cases}
\end{equation}
and 
\begin{equation}\label{LogEx2}
\f'(t)\ll \begin{cases} |t| & \text{ if } |t|<1,\\ \displaystyle{\frac{\log(2|t|)}{|t|}} & \text{ if } |t|> 1.\end{cases}
\end{equation}
\end{lem}
\begin{proof}[Proof of Lemma \ref{LogEx}]
We start by proving \eqref{LogEx1}. If $|t|\leq 1$, we use the Taylor expansion $\ex(e^{t\X})= \ex(1+t\X+ O(t^2\X^2))= 1+ O(t^2)$ since $\ex(\X)= 0$ and $|\X|\leq N$. This implies the desired estimate for $\f(t)$ when $|t|\leq 1$.

We now suppose that $|t|>1$. We will only prove the result when $t>1$, since the proof in the case $t<-1$ is similar. Let $\varepsilon>0$ be a parameter to be chosen. Then we have
$$ \pr(\X>N-\varepsilon) e^{t(N-\varepsilon)} \leq  \ex(e^{t\X}) \leq e^{tN}.$$
Choosing $\varepsilon=1/(2t)$ and using Assumption 3a we obtain
\begin{equation}\label{MomentGenerating}
\frac{e^{tN}}{(2t)^A} \ll \ex(e^{t\X}) \leq e^{tN},
\end{equation}
from which the desired estimate for $\f(t)$ follows in this case. 

Next, we establish \eqref{LogEx2}. Note that $\f$ is differentiable on $\mathbb{R}\setminus\{-1, 1\}$ and we have 
\begin{equation}\label{Derivf}
\f'(t):=\displaystyle{\frac{\ex (\X e^{t\X})}{\ex (e^{t\X})}}+ \begin{cases} 0& \text{ if } |t| <1,\\  - N & \text{ if } t> 1,\\  N & \text{ if } t< -1.\end{cases}
\end{equation}
As before, in the case $|t|<1$ the estimate of $\f'(t)$ follows from the Taylor expansions $\ex(e^{t\X})=1+O(t^2)$ and $\ex(\X e^{t\X})= \ex(\X+t\X^2+ O(t^2|\X^3|))= t\ex(\X^2)+ O(t^2)$. 

We now suppose that $t>1$, and let $\delta>0$ be a parameter to be chosen. Let $\mathcal{A}$ be the event $\X>N -\delta$, and $ \mathcal{A}^{c}$ be its complement. Then we have 
$$ \ex (\X e^{t\X})= \ex(\mathbf{1}_{\mathcal{A}} \cdot \X e^{t\X}) + \ex(\mathbf{1}_{\mathcal{A}^c} \cdot \X e^{t\X})\geq (N-\delta) \ex(\mathbf{1}_{\mathcal{A}} \cdot  e^{t\X})+ O(e^{t(N-\delta)}).$$
where $\mathbf{1}_{\mathcal{B}}$ denotes the indicator function of an event $\mathcal{B}$. Hence, using that $\ex(\mathbf{1}_{\mathcal{A}^c} \cdot e^{t\X})\leq e^{t(N-\delta)}$ we deduce
\begin{equation}\label{MomentGen2} \ex (\X e^{t\X}) \geq (N-\delta) \ex( e^{t\X})+ O(e^{t(N-\delta)}).
\end{equation}
We choose $\delta= (A+1) (\log 2t)/t$. Then, it follows from \eqref{MomentGenerating} that 
$$ e^{t(N-\delta)} = \frac{e^{tN}}{(2t)^{A+1}}\ll \frac{\ex(e^{t\X})}{t}.$$
Inserting this estimate in \eqref{MomentGen2}, and using the bound $\X\leq N$ gives 
$$ N- C\frac{\log (2t)}{t} \leq \frac{\ex (\X e^{t\X})}{\ex (e^{t\X})} \leq N,$$
for some positive constant $C$ which depends only on $A$. 
This implies the desired estimate for $\f'$ in this case. 
The proof in the case $t<-1$ follows along the same lines. 

\end{proof}
\begin{proof}[Proof of Proposition \ref{AsympLaplace}]
First, note that for $-m/2<h\le m/2$ with $h\neq 0$ we have
\begin{equation}\label{BGM}
|\gamma_m(h)|= \left| \im \left(\frac{e_m\left(h\left(\lfloor \frac{m}{2}\rfloor+1\right)\right)-1}{m\left(e_m\left(h\right)-1\right)}\right)\right|\leq \frac{1}{m|\sin(\pi h/m)|} \leq \frac{1}{2|h|},
\end{equation}
since $\sin(\pi \alpha)\geq 2\alpha$ for $0\leq \alpha\leq 1/2$.
Furthermore, it follows from \eqref{Planche2} that
\begin{equation}\label{OddEvenGM}
\gamma_m(h)=\im \left(\frac{1-e^{\pi i h}}{2\pi i h}\right) +O\left(\frac{1}{m}\right)= \begin{cases} O(\frac1m) & \text{ if } h \text{ is even},\\ -\frac{1}{\pi h}+ O(\frac1m) & \text{ if } h \text{ is odd}.
\end{cases}
\end{equation}
By the independence of the $\X(h)$ we have
$$  \log \mathcal{L}_{\X}(s)= \sum_{\substack{-m/2<h\le m/2\\ h\neq 0}}\log \ex\big(\exp\left(s\cdot\gamma_m(h) \X(h)\right)\big).$$
Using the estimate \eqref{OddEvenGM} and Lemma \ref{LogEx} we obtain
$$\sum_{\substack{-m/2<h\leq m/2\\ h\neq 0 \text{ is even }}}\log \ex\big(\exp\left(s \cdot \gamma_m(h) \X(h)\right)\big) \ll \sum_{\substack{-m/2<h\leq m/2\\ h\neq 0 \text{ is even }}} \frac{s^2}{m^2} \ll \frac{1}{m^{1/3}}. $$
We now restrict ourselves to the case $h= 2k+1$ is odd. First, it follows from \eqref{BGM} and Lemma \ref{LogEx}
that
$$  \sum_{|k|> s^2}  \log \ex\big(\exp\left(s \cdot \gamma_m(2k+1) \X(2k+1)\right)\big) \ll \sum_{|k|> s^2} \frac{s^2}{k^2} \ll 1.$$
Moreover, when $ |k|\le s^2$ we use \eqref{OddEvenGM} to get
$$
 \log \ex\big(\exp\left(s \cdot \gamma_m(2k+1) \X(2k+1)\right)\big)= \log \ex\left(\exp\left(-\frac{s}{(2k+1)\pi}\X(2k+1)\right)\right) +O\left(\frac{s}{m}\right).
$$
 Combining these estimates, and using Lemma \ref{LogEx} we obtain 
 \begin{equation}\label{ParSum}
 \log\mathcal{L}_{\X}(s)
 = \frac{2N}{\pi} s \sum_{1\leq 2k+1\leq s/\pi}\frac{1}{2k+1} + \sum_{-s^2\leq k\leq s^2} \f\left(-\frac{s}{(2 k+1)\pi}\right) +O(1).
\end{equation}
Next, we observe that
$$ \sum_{1\leq 2k+1\leq s/\pi}\frac{1}{2k+1}= \frac{1}{2}\sum_{1\leq k\leq s/2\pi}\frac{1}{k} +\log 2+O\left(\frac{1}{s}\right)= \frac{\log s}{2} + \frac{1}{2} \left( \gamma+ \log 2- \log \pi\right)+ O\left(\frac{1}{s}\right).
$$
Furthermore, by partial summation and Lemma \ref{LogEx} we get
\begin{align*}
\sum_{-s^2\leq k\leq s^2} \f\left(-\frac{s}{(2 k+1)\pi}\right) =  \int_{-s^2}^{-1}\f\left(-\frac{s}{(2 u+1)\pi}\right)du + \int_{0}^{s^2}\f\left(-\frac{s}{(2 u+1)\pi}\right)du+ O(\log^2 s). 
\end{align*}
Finally, making the change of variables $v=-s/((2 u+1)\pi)$, the main term on the right hand side of this estimate becomes
$$
 \frac{s}{2\pi} \int_{-s/\pi}^{-s/((2s^2+1)\pi)} \frac{\f(v)}{v^2}dv + \frac{s}{2\pi} \int_{s/((2s^2-1)\pi)}^{s/\pi} \frac{\f(v)}{v^2}dv = \frac{s}{2\pi}\int_{-\infty}^{\infty} \frac{\f(v)}{v^2}dv +O(\log s),
$$
by Lemma \ref{LogEx}. Inserting these estimates in \eqref{ParSum} completes the proof.
\end{proof}
\subsection{The distribution of traces of random matrices in the classical group $\mathrm{USp}_{2n}$}

Fix a positive integer $n$ and put $N=2n$. Let us endow the unitary symplectic group $G=\mathrm{USp}_{2n}$  with its Haar measure $\mu$ (which is normalized throughout), and consider
the random variable $\X:G\rightarrow[-N,N]$ that maps $M$ to $\Tr M$. Assumption 3b is
easy to check for $\X$. Indeed, let $\ell$ be an odd positive integer.
Observing that $-I_{2n}\in\mathrm{USp}_{2n}$ we get
\[\int_G(\Tr M)^\ell\mathrm{d}\mu(M)=\int_G(\Tr(-M))^\ell\mathrm{d}\mu(M)=-\int_G(\Tr M)^\ell\mathrm{d}\mu(M),\]
whence $\ex(\X^\ell)=0$. In the same way, we have
$\pr(\X>N-\varepsilon)=\pr(\X<-N+\varepsilon)$ for every $\varepsilon>0$. Let us now verify Assumption 3a.

\begin{lem}\label{traceUSp}
There exists a positive real number $c_n$ such that for every
$\varepsilon\in(0,2]$, one has
$\pr(\X>N-\varepsilon)\ge c_n\varepsilon^{n(2n+1)/2}$.
\end{lem}

\begin{proof}
Put
\[\mathcal{A}=\Bigl\{\underline{\theta}=(\theta_1,\cdots,\theta_n)\in[0,\pi]^n\ \Big|\ \sum_{h=1}^n2\cos\theta_h>N-\varepsilon\Bigr\}=\Bigl\{\underline{\theta}\in[0,\pi]^n\ \Big|\ \sum_{h=1}^n\sin^2\frac{\theta_h}{2}<\frac{\varepsilon}{4}\Bigr\}\]
and $\mathcal{B}=\{\underline{t}\in\mathbb{R}_+^n\ |\ t_1^2+\cdots+t_n^2<1\}$.
The Weyl integration formula (see for example page 117 of \cite{Co}) gives
\[\pr(\X>2n-\varepsilon)=\frac{2^{n^2}}{n!\pi^n}\int_\mathcal{A}\prod_{1\le j<k\le n}(\cos\theta_k-\cos\theta_j)^2\prod_{h=1}^n\sin^2\theta_h\mathrm{d}\theta_1\cdots\mathrm{d}\theta_n.\]
Let us remark that if $\underline{\theta}\in\mathcal{A}$ then
\[\prod_{h=1}^n\cos^2\frac{\theta_h}{2}=\prod_{h=1}^n\Bigl(1-\sin^2\frac{\theta_h}{2}\Bigr)\ge1-\sum_{h=1}^n\sin^2\frac{\theta_h}{2}>1-\frac{\varepsilon}{4}.\]
We infer that
\begin{align*}
\pr(\X>N-\varepsilon)&=\frac{2^{n(2n+1)}}{n!\pi^n}\int_\mathcal{A}\prod_{j<k}\Bigl(\sin^2\frac{\theta_j}{2}-\sin^2\frac{\theta_k}{2}\Bigr)^2\prod_{h=1}^n\sin^2\frac{\theta_h}{2}\cos^2\frac{\theta_h}{2}\mathrm{d}\theta_1\cdots\mathrm{d}\theta_n\\
&\ge\frac{2^{n(2n+1)}}{n!\pi^n}\sqrt{1-\frac{\varepsilon}{4}}\int_\mathcal{A}\prod_{j<k}\Bigl(\sin^2\frac{\theta_j}{2}-\sin^2\frac{\theta_k}{2}\Bigr)^2\prod_{h=1}^n\sin^2\frac{\theta_h}{2}\cos\frac{\theta_h}{2}\mathrm{d}\theta_1\cdots\mathrm{d}\theta_n\\
&=\frac{2^n\varepsilon^{n(2n+1)/2}}{n!\pi^n}\sqrt{1-\frac{\varepsilon}{4}}\int_\mathcal{B}\prod_{1\le j<k\le n}(t_j^2-t_k^2)^2\prod_{h=1}^nt_h^2\mathrm{d}t_1\cdots\mathrm{d}t_n,\\
\end{align*}
where we use the change of variables
$\displaystyle t_h=\frac{2}{\sqrt{\varepsilon}}\sin\frac{\theta_h}{2}$. Whence
the result.
\end{proof}

\section{Completing the proof of the upper bound in Theorem \ref{Main}: Proof of Theorem \ref{MomentsMaxTail}\label{MomentsSec}}
In this section, we assume that $\mathcal{F}=\{\ph_a\}_{a\in \Omega_m}$ is a family of $m$-periodic complex valued functions satisfying Assumptions 2 and 3. We start by proving the following lemma which follows from combining Assumption 3 with Lemma \ref{BOUNDRAND}.
\begin{lem}\label{BoundMomentsyz}
Let $m$ be a large positive integer, and $1\leq y <z\leq m/2$ be real numbers. Let $\{\ph_a\}_{a\in \Omega_m}$ be a family of $m$-periodic complex valued functions satisfying Assumptions 2 and 3. Let $\{c(h)\}_{h\in \mathbb{Z}^*}$ be a sequence of complex numbers such that $|c(h)|\leq c_0/|h|$, where $c_0$ is a positive constant. Then, for all positive integers  $k\leq (\log m)/(5\log\log m)$ we have   
$$
\frac{1}{|\Omega_m|}\sum_{a\in \Omega_m}\left|\sum_{y\le |h|<z} c(h) \ta(h)\right|^{2k} \ll \left(\frac{16(c_0N)^2k}{y}\right)^{k}+\frac{(4C_1c_0\log m)^{2k}}{m^{\eta}}.
$$
\end{lem}
\begin{proof} Expanding the moments and using Assumptions 2 and 3, we obtain
\begin{align*}
&\frac{1}{|\Omega_m|}\sum_{a\in \Omega_m}\left|\sum_{y\le |h|<z} c(h) \ta(h)\right|^{2k}\\
&= \sum_{\substack{y\le |h_1|, \dots, |h_k|<z\\ y\le |r_1|, \dots, |r_k|<z}} c(h_1)\cdots c(h_k) \overline{c(r_1)\cdots c(r_k)} \ \frac{1}{|\Omega_m|}\sum_{a\in \Omega_m}\prod_{u=1}^k\ta(h_u)\prod_{v=1}^{k}\ta(r_v)\\
&=\sum_{\substack{y\le |h_1|, \dots, |h_k|<z\\ y\le |r_1|, \dots, |r_k|<z}} c(h_1)\cdots c(h_k) \overline{c(r_1)\cdots c(r_k)} \ \ex\left(\prod_{u=1}^k \X(h_{u})\prod_{v=1}^{k}\X(r_v)\right) + E_{k}(y, z),
\end{align*}
where the error term satisfies
$$ E_{k}(y, z) \ll \frac{C_1^{2k}}{m^{\eta}} \Big(\sum_{y\leq |h|<z} |c(h)|\Big)^{2k} \leq \frac{C_1^{2k}}{m^{\eta}} \Big(\sum_{y\leq |h|<z} \frac{c_0}{|h|}\Big)^{2k} \ll \frac{(4C_1c_0\log m)^{2k}}{m^{\eta}},$$
by Assumption 3. 
The result follows upon noting that
\begin{align*}
&\sum_{\substack{y\le |h_1|, \dots, |h_k|<z\\ y\le |r_1|, \dots, |r_k|<z}} c(h_1)\cdots c(h_k) \overline{c(r_1)\cdots c(r_k)} \ \ex\left(\prod_{u=1}^k \X(h_{u})\prod_{v=1}^{k}\X(r_v)\right)\\
&=\ex\left(\left|\sum_{y\le |h|<z} c(h) \X(h)\right|^{2k}\right) \leq \left(\frac{16(c_0N)^2k}{y}\right)^{k},
\end{align*}
by Lemma \ref{BOUNDRAND}. 
\end{proof}

We will deduce Theorem \ref{MomentsMaxTail} from the following results, which generalize Propositions 6.1 and 6.2 of \cite{La}. In both results we assume that $\{\ph_a\}_{a\in \Omega_m}$ is a family of $m$-periodic complex valued functions satisfying Assumptions 2 and 3. 
\begin{pro}\label{MomentsMaxTail1}
Let $m$ be large, and $k$ be an integer such that $10^5N^2<k\leq (\log m)/(5\log\log m)$. Let $S$ be a non-empty subset of $[0, 1)$, and put $y=10^5 N^2 k$.  Then we have 
$$ 
\frac{1}{|\Omega_m|}\sum_{a\in \Omega_m} \max_{\alpha\in S}\left|\sum_{y\le |h|<k^2} \frac{e\left(\alpha h\right)- 1}{h} \ta(h)\right|^{2k} \ll e^{-4k}.
$$
\end{pro}
\begin{pro}\label{MomentsMaxTail2}
Let $m$ be a large positive integer, and $k$ be an integer such that $3\leq k\leq (\log m)/(50\log\log m)$. Let $S$ be a non-empty finite subset of $[0, 1)$. 
Then we have $$ 
\frac{1}{|\Omega_m|}\sum_{a\in \Omega_m} \max_{\alpha\in S}\left|\sum_{k^2 \le |h|< m/2} \frac{e\left(\alpha h\right)- 1}{h} \ta(h)\right|^{2k} \ll e^{-4k} + \frac{|S|(2C_1\log m)^{8k}}{m^{\eta}}.
$$
\end{pro}
We start by proving Proposition \ref{MomentsMaxTail1}, as its proof is simpler since the inner sum over $|h|$ is short.
\begin{proof}[Proof of  Proposition \ref{MomentsMaxTail1}] 
Let
$\mathcal{A}_k= \{b/k^4 : 1\leq b\leq k^4\}.$ Then for all $\alpha \in S$, there exists $\beta_{\alpha}\in \mathcal{A}_k$ such that 
$|\alpha-\beta_{\alpha}|\leq 1/k^4$. In this case we have $e(\alpha h)= e(\beta_{\alpha} h)+ O(h/k^4),$
 and hence 
$$\sum_{y\le |h|<k^2} \frac{e\left(\alpha h\right)- 1}{h} \ta(h)= \sum_{y\le |h|<k^2} \frac{e(\beta_{\alpha} h)- 1}{h} \ta(h) +O\left(\frac{1}{k^2}\right). 
 $$
 Therefore, using the elementary inequality $|x+y|^{2k}\leq 2^{2k} (|x|^{2k}+|y|^{2k})$ we deduce that
 \begin{equation}\label{Transition}
\begin{aligned}
&\max_{\alpha\in S}\left|\sum_{y\le |h|<k^2} \frac{e\left(\alpha h\right)- 1}{h} \ta(h)\right|^{2k}\leq 2^{2k} \max_{\alpha\in \mathcal{A}_k}\left|\sum_{y\le |h|<k^2} \frac{e\left(\alpha h\right)- 1}{h} \ta(h)\right|^{2k}+ \left(\frac{c_1}{k^2}\right)^{2k}
\\
& \leq 2^{2k} \sum_{\alpha\in \mathcal{A}_k}\left|\sum_{y\le |h|<k^2} \frac{e\left(\alpha h\right)- 1}{h} \ta(h)\right|^{2k}+ \left(\frac{c_1}{k^2}\right)^{2k},
 \end{aligned}
 \end{equation}
 for some positive constant $c_1$. Thus, it follows from Lemma \ref{BoundMomentsyz}  that in this case we have
 \begin{equation}\label{SecondCaseBound}
\begin{aligned}
&\frac{1}{|\Omega_m|}\sum_{a\in \Omega_m} \max_{\alpha\in S}\left|\sum_{y\le |h|<k^2} \frac{e\left(\alpha h\right)- 1}{h} \ta(h)\right|^{2k}\\
& \leq 2^{2k} \sum_{\alpha\in \mathcal{A}_k} \frac{1}{|\Omega_m|} \sum_{a\in \Omega_m}\left|\sum_{y\le |h|<k^2} \frac{e\left(\alpha h\right)- 1}{h} \ta(h)\right|^{2k} + \left(\frac{c_1}{k^2}\right)^{2k}\\
&\ll k^4 2^{2k}\left(\left(\frac{64 N^2 k}{y}\right)^{k}+\frac{(8C_1\log m)^{2k}}{m^{\eta}}\right)+\left(\frac{c_1}{k^2}\right)^{2k}\ll e^{-4k},
\end{aligned}
\end{equation}
which completes the proof.
\end{proof}
\begin{proof}[Proof of Proposition \ref{MomentsMaxTail2}]
Since the inner sum over $|h|$ is long in this case, we shall split it into dyadic intervals. Let $J_1= \lfloor \log(k^2)/\log 2\rfloor$ and $J_2= \lfloor \log(m/2)/\log 2\rfloor$. We define $z_{J_1}:= k^2$, $z_{J_2+1}:=m/2$, and $z_j:=2^{j}$ for $J_1+1\leq j\leq J_2$. Then, using H\"older's inequality we obtain
\begin{equation}\label{DyadicSums}
\begin{aligned}
& \left|\sum_{k^2\le |h|<m/2} \frac{e\left(\alpha h\right)- 1}{h} \ta(h)\right|^{2k}= \left|\sum_{J_1\leq j\leq J_2}\frac{1}{j^2} \cdot \left(j^2\sum_{z_j\le |h|<z_{j+1}} \frac{e\left(\alpha h\right)- 1}{h} \ta(h)\right)\right|^{2k}\\
& \ \ \ \ \ \ \ \leq \left(\sum_{J_1\leq j\leq J_2} \frac{1}{j^{4k/(2k-1)}}\right)^{2k-1} \left(\sum_{J_1\leq j\leq J_2} j^{4k} \left|\sum_{z_j\le |h|<z_{j+1}} \frac{e\left(\alpha h\right)- 1}{h} \ta(h)\right|^{2k}\right)\\
&  \ \ \ \ \ \ \ \leq \left(\frac{c_2}{\log k}\right)^{2k+1} \sum_{J_1\leq j\leq J_2} j^{4k} \left|\sum_{z_j\le |h|<z_{j+1}} \frac{e\left(\alpha h\right)- 1}{h} \ta(h)\right|^{2k},
\end{aligned}
\end{equation}
for some constant $c_2>0$. Therefore, this reduces the problem to bounding the corresponding moments over each dyadic interval $[z_j, z_{j+1}]$, namely
$$\frac{1}{|\Omega_m|}\sum_{a\in \Omega_m} \max_{\alpha\in S}\left|\sum_{z_j\le |h|<z_{j+1}} \frac{e\left(\alpha h\right)- 1}{h} \ta(h)\right|^{2k}.
$$
We shall consider two cases, depending on whether $j$ is large in terms of $|S|$. First, if $4^j\geq  |S|$ then by Lemma \ref{BoundMomentsyz} we have
\begin{equation}\label{Largej}
\begin{aligned}
&\frac{1}{|\Omega_m|}\sum_{a\in \Omega_m} \max_{\alpha\in S}\left|\sum_{z_j\le |h|<z_{j+1}} \frac{e\left(\alpha h\right)- 1}{h} \ta(h)\right|^{2k}\\
& \leq  \sum_{\alpha\in S}\frac{1}{|\Omega_m|}\sum_{a\in \Omega_m} \left|\sum_{z_j\le |h|<z_{j+1}} \frac{e\left(\alpha h\right)- 1}{h} \ta(h)\right|^{2k} \ll 4^j\left(\frac{64 N^2 k}{2^j}\right)^{k}+\frac{|S|(8C_1\log m)^{2k}}{m^{\eta}}.
\end{aligned}
\end{equation}
since $z_j \geq 2^j$ for $J_1\leq j\leq J_2$.  We now suppose that $4^j<|S|$, and let $\mathcal{B}_j= \{b/4^j : 1\leq b\leq 4^j\}$. Then for all $\alpha \in S$ there exists $\beta_{\alpha}\in \mathcal{B}_j$ such that 
$|\alpha-\beta_{\alpha}|\leq 1/4^j$. In this case we have $e(\alpha h)= e(\beta_{\alpha} h)+ O(h/4^j),$
 and hence we obtain
 $$ \sum_{z_j\le |h|<z_{j+1}} \frac{e\left(\alpha h\right)- 1}{h} \ta(h)= \sum_{z_j\le |h|<z_{j+1}} \frac{e(\beta_{\alpha} h)- 1}{h} \ta(h) +O\left(\frac{1}{2^j}\right),
 $$
 since $z_{j+1}\asymp z_j\asymp 2^j$.  Therefore, similarly to \eqref{Transition} we derive
\begin{align*}
&\max_{\alpha\in S}\left|\sum_{z_j\le |h|<z_{j+1}} \frac{e\left(\alpha h\right)- 1}{h} \ta(h)\right|^{2k}\\
& \leq 2^{2k} \max_{\alpha\in \mathcal{B}_j}\left|\sum_{z_j\le |h|<z_{j+1}} \frac{e\left(\alpha h\right)- 1}{h} \ta(h)\right|^{2k}+ \left(\frac{c_3}{2^j}\right)^{2k},
 \end{align*}
 for some positive constant $c_3$. Thus, appealing to Lemma \ref{BoundMomentsyz} we get
 \begin{equation}\label{Smallj}
 \begin{aligned}
&\frac{1}{|\Omega_m|}\sum_{a\in \Omega_m} \max_{\alpha\in S}\left|\sum_{z_j\le |h|<z_{j+1}} \frac{e\left(\alpha h\right)- 1}{h} \ta(h)\right|^{2k} \\
&\leq 2^{2k}\sum_{\alpha\in \mathcal{B}_j} \frac{1}{|\Omega_m|}\sum_{a\in \Omega_m}\left|\sum_{z_j\le |h|<z_{j+1}} \frac{e\left(\alpha h\right)- 1}{h} \ta(h)\right|^{2k}+ \left(\frac{c_3}{2^j}\right)^{2k}\\
& \ll 4^j\left(\frac{2^8 N^2 k}{2^j}\right)^{k}+\frac{|S|(16C_1\log m)^{2k}}{m^{\eta}},
 \end{aligned}
 \end{equation}
 since $|\mathcal{B}_j|=4^j <|S|.$ Combining \eqref{Largej} and \eqref{Smallj} we deduce that in all cases we have
 $$
 \frac{1}{|\Omega_m|}\sum_{a\in \Omega_m} \max_{\alpha\in S}\left|\sum_{z_j\le |h|<z_{j+1}} \frac{e\left(\alpha h\right)- 1}{h} \ta(h)\right|^{2k}  \ll 4^j\left(\frac{2^8N^2k}{2^j}\right)^{k}+\frac{|S|(16C_1\log m)^{2k}}{m^{\eta}}.
 $$
Inserting this bound in \eqref{DyadicSums} gives
\begin{equation}\label{ThirdCaseBound}
\begin{aligned}
&\frac{1}{|\Omega_m|}\sum_{a\in \Omega_m} \max_{\alpha\in S}\left|\sum_{k^2\le |h|<m/2} \frac{e\left(\alpha h\right)- 1}{h} \ta(h)\right|^{2k}\\
 &\ll \left(\frac{c_4}{\log k}\right)^{2k+1} k^k \sum_{J_1\leq j\leq J_2} 4^j\left(\frac{j^4}{2^{j}}\right)^{k}+ \frac{|S|(2C_1\log m)^{8k}}{m^{\eta}}\\
 &\ll e^{-4k} + \frac{|S|(2C_1\log m)^{8k}}{m^{\eta}},
\end{aligned}
\end{equation}
for some positive constant $c_4$, where the last estimate follows since $j^4\leq 2^{j/4}$ for $j$ large enough, and $2^{J_1}\asymp k^2$. This completes the proof.
\end{proof}
Finally, we deduce Theorem \ref{MomentsMaxTail}.

\begin{proof}[Proof of Theorem \ref{MomentsMaxTail}]  By
Minkowski's inequality we have
\begin{align*}
&\left(\sum_{a\in \Omega_m} \max_{\alpha\in S}\left|\sum_{y\le |h| <m/2} \frac{e\left(\alpha h\right)- 1}{h} \ta(h)\right|^{2k}\right)^{1/2k}\\
 \leq 
 &  \ \left(\sum_{a\in \Omega_m} \max_{\alpha\in S}\left|\sum_{y\le |h|<k^2} \frac{e\left(\alpha h\right)- 1}{h} \ta(h)\right|^{2k}\right)^{1/2k}\\ 
& \ \ \ \ \ \ \ \ \ \ \ \ \ \ \ \ \ \ \ \ \ \ \ \ \ \ \ \ \ \ +\left(\sum_{a\in \Omega_m} \max_{\alpha\in S}\left|\sum_{k^2\le |h|< m/2} \frac{e\left(\alpha h\right)- 1}{h} \ta(h)\right|^{2k}\right)^{1/2k}.
\end{align*}
The result follows upon using Propositions \ref{MomentsMaxTail1} and \ref{MomentsMaxTail2}.
\end{proof}


\section{Proof  of Theorem \ref{Lower}\label{LowerSec}}
In this section we shall investigate the distribution of the partial sums $\sum_{0\leq n\leq x} \ph_a(n)$ in the special case $x=m/2$, where $\mathcal{F}=\{\ph_a\}_{a\in \Omega_m}$ is a family of $m$-periodic complex valued functions satisfying Assumptions 2 and 3. For a real number $t$, we define
\begin{equation}\label{LargeSumt}
\Psi_{\mathcal{F}}(t):=\frac{1}{|\Omega_m|} \Bigg| \Bigg \{a\in \Omega_m: \frac{1}{\sqrt{m}}  \im \sum_{0\leq n\leq m/2} \ph_a(n)> t\Bigg\} \Bigg|,
\end{equation}
and similarly we write $\Psi_{\mathcal{F}}^-(t)$ for the proportion
of $a\in\Omega_m$ such that $\frac{1}{\sqrt{m}}\im\sum_{0\leq n\leq m/2}\ph_a(n)<-V$.
We will prove the following result from which Theorem \ref{Lower} follows.
\begin{thm}\label{LowerB}
Let $m$ be large and $\mathcal{F}=\{\ph_a\}_{a\in \Omega_m}$ be a family of $m$-periodic complex valued functions satisfying Assumptions 2 and 3. Uniformly for $V$ in the range $1\leq V\leq \frac{N}{\pi} (\log\log m- 2\log\log\log m-B)$ we have
$$
\Psi_{\mathcal{F}}(V)= \exp\left(-A_0 \exp\left(\frac{\pi}{N} V\right) \left(1+O\left(V e^{-\pi V/(2N)}\right)\right)\right).
$$
Furthermore, in the same range of $V$ we also have 
$$\Psi_{\mathcal{F}}^-(V)= \exp\left(-A_0 \exp\left(\frac{\pi}{N} V\right) \left(1+O\left(V e^{-\pi V/(2N)}\right)\right)\right).$$
\end{thm}
Recall from \eqref{EqualityB} that
$$
\frac{1}{\sqrt{m}}\im \sum_{0\leq n\leq m/2} \ph_a(n)=\sum_{\substack{-m/2<h\le m/2\\ h\neq 0}}\gamma_m(h) \ta(h).
$$
In order to prove Theorem \ref{LowerB}, we will show that the moment generating function of the sum $\sum_{\substack{-m/2<h\le m/2\\ h\neq 0}}\gamma_m(h) \ta(h)$ (after removing a ``small'' set of ``bad'' points $a$) is very close to the moment generating function of the probabilistic random model $\sum_{\substack{-m/2<h\le m/2\\ h\neq 0}}\gamma_m(h) \X(h)$, which we already estimated in Proposition \ref{AsympLaplace}.
\begin{pro}\label{Laplace}
Let $m$ be large. There exists a set $\mathcal{E}_m\subset \Omega_m$ with cardinality $|\mathcal{E}_m|\leq m^{-1/10} |\Omega_m|$ such that for all complex numbers $s$ with $N|s|\leq (\log m)/(50 \log \log m)^2$ we have
$$
 \frac{1}{|\Omega_m|} \sum_{a\in \Omega_m\setminus \mathcal{E}_m} \exp\left(s \cdot  \sum_{\substack{-m/2<h\le m/2\\ h\neq 0}}\gamma_m(h) \ta(h)\right)= \mathcal{L}_{\X}(s)+ O\left(\exp\left(-\frac{\log m}{20\log\log m}\right)\right).
$$
\end{pro}
\begin{proof}
Let $\mathcal{E}_m$ be the set of $a\in \Omega_m$ such that 
$$ \left|\sum_{\substack{-m/2<h\le m/2\\ h\neq 0}}\gamma_m(h) \ta(h)\right| \geq 4N \log\log m.$$
Using Assumption 2 together with the bound \eqref{BGM} we get 
$$
\left|\sum_{\substack{-m/2<h\le m/2\\ h\neq 0}}\gamma_m(h) \ta(h)\right|
 \leq 3N \log\log m + \left|\sum_{(\log m)^2<|h|<m/2}\gamma_m(h) \ta(h)\right|,
$$
if $m$ is sufficiently large. Therefore, it follows from Lemma \ref{BoundMomentsyz} that for $r=\lfloor \log m/(10\log\log m)\rfloor$ we have
\begin{equation}\label{BoundE}
\begin{aligned}
|\mathcal{E}_m| 
&\leq \Big|\Big\{a\in \Omega_m: \big|\sum_{(\log m)^2<|h|<m/2}\gamma_m(h) \ta(h)\big|\geq N \log\log m\Big\}\Big|\\
& \leq (N\log\log m)^{-2r} \sum_{a\in \Omega_m} \left|\sum_{(\log m)^2<|h|<m/2}\gamma_m(h) \ta(h)\right|^{2r} \\
& \ll m^{-1/10} |\Omega_m|.
\end{aligned}
\end{equation}

Let $L=\lfloor \log m/(20\log\log m)\rfloor$. Then we have
\begin{equation}\label{TaylorLaplace}
\begin{aligned}
&\frac{1}{|\Omega_m|}\sum_{a\in \Omega_m\setminus \mathcal{E}_m}\exp\left(s \cdot \sum_{\substack{-m/2<h\le m/2\\ h\neq 0}}\gamma_m(h) \ta(h)\right)\\
&= \sum_{k=0}^L \frac{s^k}{k!} \frac{1}{|\Omega_m|}\sum_{a\in \Omega_m\setminus \mathcal{E}_m} \left(\sum_{\substack{-m/2<h\le m/2\\ h\neq 0}}\gamma_m(h) \ta(h)\right)^k+E_1\\
\end{aligned}
\end{equation}
where
$$ E_1\ll \sum_{k>L} \frac{|s|^k}{k!} (4N\log\log m)^k\leq \sum_{k>L} \left(\frac{15 N |s|\log\log m}{L}\right)^k \ll e^{-L} $$
by Stirling's formula and our assumption on $s$. Furthermore, note that 
$$ \sum_{\substack{-m/2<h\le m/2\\ h\neq 0}}|\gamma_m(h) \ta(h)| \leq \frac{N}{2}\sum_{1\leq |h|\leq m/2} \frac{1}{|h|}\leq 3N \log m,$$ 
if $m$ is sufficiently large. 
Therefore, it follows from equation \eqref{BoundE}  that for all integers $0\leq k\leq L$ we have 
\begin{equation}\label{MoAsymp}
\begin{aligned}
&\frac{1}{|\Omega_m|}\sum_{a\in \Omega_m\setminus \mathcal{E}_m} \left(\sum_{\substack{-m/2<h\le m/2\\ h\neq 0}}\gamma_m(h) \ta(h)\right)^k\\
&=\frac{1}{|\Omega_m|}\sum_{a\in \Omega_m} \left( \sum_{\substack{-m/2<h\le m/2\\ h\neq 0}}\gamma_m(h) \ta(h)\right)^k +O\left(m^{-1/10}(3N\log m)^{k}\right)\\
&=\ex\left(\Big( \sum_{\substack{-m/2<h\le m/2\\ h\neq 0}}\gamma_m(h) \X(h)\Big)^k\right) +O\left(m^{-1/25}\right),
\end{aligned}
\end{equation}
where the last equality follows from expanding the moments and using Assumption 3 as in the proof of Lemma \ref{BoundMomentsyz}. 

\noindent Furthermore, it follows from equation \eqref{BGM}, Lemma \ref{BOUNDRAND} and Stirling's formula that  
$$ \sum_{k>L} \frac{|s|^k}{k!} \ex\left(\Big|\sum_{\substack{-m/2<h\le m/2\\ h\neq 0}}\gamma_m(h) \X(h)\Big|^k\right) \ll \sum_{k>L} \left(\frac{15N|s|\log k}{k}\right)^k\ll \sum_{k>L} \left(\frac{15N|s|\log L}{L}\right)^k\ll e^{-L}.$$
Finally, combining this bound with \eqref{TaylorLaplace} and \eqref{MoAsymp},  we derive
\begin{align*}
&\frac{1}{|\Omega_m|}\sum_{a\in \Omega_m\setminus \mathcal{E}_m}\exp\left(s \cdot \sum_{\substack{-m/2<h\le m/2\\ h\neq 0}}\gamma_m(h) \ta(h)\right)\\
&=  \sum_{k=0}^L \frac{s^k}{k!} \ex\left(\Big(\sum_{\substack{-m/2<h\le m/2\\ h\neq 0}}\gamma_m(h) \X(h)\Big)^k\right) +O\left(e^{-L}+ m^{-1/25}e^{|s|}\right)\\
&= \mathcal{L}_{\X}(s)+  O\left(e^{-L}\right),
\end{align*}
as desired.
\end{proof}
Using the saddle-point method and Propositions \ref{AsympLaplace} and \ref{Laplace}, we prove Theorem \ref{LowerB}.
\begin{proof}[Proof Theorem \ref{LowerB}]

 Let $ \mathcal{E}_m$ be the set in the statement of Proposition \ref{Laplace}, and $\widetilde{\Psi_{\mathcal{F}}}(t)$ be the proportion of $a\in \Omega_m \setminus \mathcal{E}_m$ such that $\sum_{\substack{-m/2<h\le m/2\\ h\neq 0}} \gamma_m(h)\ta(h)>t$. Then, it follows from equation \eqref{EqualityB} and Proposition \ref{Laplace} that
$$ \Psi_{\mathcal{F}}(t)= \widetilde{\Psi_{\mathcal{F}}}(t) +O\left(m^{-1/10}\right).$$
Furthermore, it follows from Propositions \ref{Laplace} and \ref{AsympLaplace} that for all positive real numbers $s$ such that $2N\leq Ns\leq (\log m)/(50\log \log m)^2$ we have 
\begin{equation}\label{EstDistri}
\begin{aligned}
\int_{-\infty}^{\infty} e^{st } \widetilde{\Psi_{\mathcal{F}}}(t) dt & 
= \frac{1}{|\Omega_m|}\sum_{a\in \Omega_m\setminus \mathcal{E}_m}\int_{-\infty}^{\sum_{\substack{-m/2<h\le m/2\\ h\neq 0}} \gamma_m(h) \ta(h)} e^{st }  dt\\
& = \frac{1}{s|\Omega_m|} \sum_{a\in \Omega_m\setminus \mathcal{E}_m} \exp\left(s \cdot \sum_{\substack{-m/2<h\le m/2\\ h\neq 0}}\gamma_m(h) \ta(h)\right)\\
&=\exp\left(\frac{N}{\pi}s\log s+ B_0 s +O(\log^2 s)\right).
\end{aligned} 
\end{equation}
The result trivially holds if $V$ is small, so we might assume that $V$ is a sufficiently large real number such that $V\leq (N/\pi)(\log\log m-2\log\log\log m-B)$, where $B=\log(N/\pi)+8 -B_0\pi/N$. We shall choose $s$ (the saddle point) such that 
\begin{equation}\label{Saddle}
\left(\frac{N}{\pi} s\log s+B_0 s-sV\right)'=0 \Longleftrightarrow s= \exp\left(\frac{\pi}{N}V-\frac{\pi}{N}B_0-1\right).
\end{equation}
Let $0<\varepsilon<1$ be a small parameter to be chosen, and put $S=s e^{\varepsilon}$. Then, it follows from \eqref{EstDistri} that 
\begin{align*}
\int_{V+N\varepsilon/\pi}^{\infty} e^{st } \widetilde{\Psi_{\mathcal{F}}}(t)dt 
&\leq \exp\left(s(1-e^{\varepsilon})(V+ N\varepsilon/\pi)\right)\int_{V+N\varepsilon/\pi}^{\infty} e^{St }\widetilde{\Psi_{\mathcal{F}}}(t)dt\\
&\leq  \exp\left(s(1-e^{\varepsilon})(V+N\varepsilon/\pi)+ \frac{N}{\pi}s e^{\varepsilon} \log s+ \frac{N}{\pi} s e^{\varepsilon} \varepsilon +B_0 s e^{\varepsilon} +O(\log^2 s)\right) \\
&=  \exp\left(\frac{N}{\pi}s\log s+B_0 s+ \frac{N}{\pi} s(1+\varepsilon-e^{\varepsilon}) +O(\log^2 s)\right).
\end{align*}
Therefore, choosing $\varepsilon= C_0(\log s)/\sqrt{s}$ for a suitably large constant $C_0$ and using \eqref{EstDistri} we obtain
$$
\int_{V+N\varepsilon/\pi}^{\infty} e^{st } \widetilde{\Psi_{\mathcal{F}}}(t)dt \leq  e^{-V^2}\int_{-\infty}^{\infty} e^{st } \widetilde{\Psi_{\mathcal{F}}}(t)dt. 
$$
A similar argument shows that
$$
\int_{-\infty}^{V-N\varepsilon/\pi} e^{st} \widetilde{\Psi_{\mathcal{F}}}(t)dt \leq  e^{-V^2}\int_{-\infty}^{\infty} e^{st } \widetilde{\Psi_{\mathcal{F}}}(t)dt. 
$$
Combining these bounds with \eqref{EstDistri} gives
\begin{equation}\label{EstDistri2}
\int_{V-N\varepsilon/\pi}^{V+N\varepsilon/\pi} e^{st } \widetilde{\Psi_{\mathcal{F}}}(t)dt =  \exp\left(\frac{N}{\pi} s\log s+ B_0 s +O(\log^2 s)\right).
\end{equation}
Furthermore, since $\widetilde{\Psi_{\mathcal{F}}}(t)$ is non-increasing as a function of $t$ we can bound the above integral as follows
$$ 
e^{sV+O(s\varepsilon)}\widetilde{\Psi_{\mathcal{F}}}(V+N\varepsilon/\pi) \leq \int_{V-N\varepsilon/\pi}^{V+N\varepsilon/\pi} e^{st } \widetilde{\Psi_{\mathcal{F}}}(t)dt\leq  e^{sV+O(s\varepsilon)} \widetilde{\Psi_{\mathcal{F}}}(V-N\varepsilon/\pi).
$$ 
Inserting these bounds in \eqref{EstDistri2} and using the definition of $s$ in terms of $V$, we obtain
$$\widetilde{\Psi_{\mathcal{F}}}(V+N\varepsilon/\pi) \leq \exp\left(- \frac{N}{\pi} \exp\left(\frac{\pi}{N}V-\frac{\pi}{N}B_0-1\right)\big(1+O( \varepsilon)\big)\right) \leq \widetilde{\Psi_{\mathcal{F}}}(V-N\varepsilon/\pi),
$$
and thus 
$$ 
\widetilde{\Psi_{\mathcal{F}}}(V)= \exp\left(- \frac{N}{\pi} \exp\left(\frac{\pi}{N}V-\frac{\pi}{N}B_0-1\right)\left(1 + O\left(Ve^{-\pi V/(2N)}\right)\right)\right),
$$
as desired.

\end{proof}
We end this section by proving Corollary \ref{Main3}. By Theorem \ref{LowerB}, it follows that there are  $\gg |\Omega_m| m^{-1/\log\log m}$ elements $a\in \Omega_m$ such that
$$ \left|\sum_{0\leq n\leq m/2} \ph_a(n)\right|\ge \left(\frac{N}{\pi}+o(1)\right) \sqrt{m}\log\log m.$$
 Hence, in order to deduce Corollary \ref{Main3} it suffices to show that $|\Omega_m|$ is larger than a multiple of $m$. 

\begin{lem}\label{SizeOmegam} Let $\mathcal{F}=\{\ph_a\}_{a\in \Omega_m}$ be a family of $m$-periodic complex valued functions satisfying Assumptions 2 and 3. Then, we must have $|\Omega_m|\gg m$, where the implied constant depends only on the involved constants in the Assumptions.
\end{lem}

\begin{proof}
Set $J=(-m/2,m/2]\cap\mathbb{Z}\smallsetminus\{0\}$. Let us recall the
following elementary result: if $M$ is a symmetric real matrix, then
$(\Tr M)^2\le(\mathrm{rk}M)\Tr(M^2)$; one sees this by applying the
Cauchy-Schwarz inequality to the non-zero eigenvalues of $M$. We use it with
$M=[b_{hj}]_{h,j}=\ ^\mathrm{t}\!LL$ where $L=[\ta(h)]_{a,h}$ (here the index $a$
is in $\Omega_m$ and $h,j$ are in $J$). Putting $\beta=\ex(\X^2)$, Assumption 3
with $k=2$ gives
\[
b_{hh}=\sum_{a\in\Omega_m}\ta(h)^2=|\Omega_m|\beta+O\Bigl(\frac{|\Omega_m|}{\sqrt{m}}\Bigr)
\]
for every $h\in J$ and
\[
b_{hj}=\sum_{a\in\Omega_m}\ta(h)\ta(j)=O\Bigl(\frac{|\Omega_m|}{\sqrt{m}}\Bigr)
\]
for every distinct $h,j$. We deduce the estimates
\[
\Tr M=\sum_{h\in J}b_{hh}=|\Omega_m|(\beta m+O(\sqrt{m}))\ ,\quad\Tr(M^2)=\sum_{h,j}b_{hj}^2=O(m|\Omega_m|^2)\ .
\]
One gets
\[|\Omega_m|^2(\beta^2m^2+O(m^{3/2}))=(\Tr M)^2\le(\mathrm{rk}M)\Tr(M^2)\ll m|\Omega_m|^2\mathrm{rk}M.\]
Whence $m\ll \mathrm{rk}M\le\mathrm{rk}L\le|\Omega_m|$.
\end{proof}


\section{An example with very large partial sums \label{HugeSec}}

In this section we shall prove Proposition \ref{HugeFamily}. Let $m\ge7$ be an
integer. Put $r=\lfloor3\log m/\log2\rfloor$ and
$\displaystyle P=\sum_{k=1}^rX^{2k-1}$. Take a finite field $\Omega_m$ with $2^r$
elements and $\psi:\Omega_m\rightarrow\{-1,1\}$ a non-trivial additive
character. By Weil's theorem \cite{Weil} one has
\[\Bigl|\sum_{a\in\Omega_m}\psi(P(a))\Bigr|\le(2r-2)2^{r/2}\textrm{, so }|\{a\in\Omega_m\ |\ \psi(P(a))=1\}|\ge\frac{2^r}{2}-(r-1)2^{r/2}>\frac{m}{2}\]
and the same is true for $|\{a\in\Omega_m\ |\ \psi(P(a))=-1\}|$. Putting
$J=(-m/2,m/2]\cap\mathbb{Z}$, we can therefore choose distinct elements
$(\alpha_h)_{h\in J}$ of $\Omega_m^\times$ such that $\psi(P(\alpha_h))=1$ if
$h\ge1$ and $\psi(P(\alpha_h))=-1$ if $h\le0$. For every $a\in\Omega_m$, we
define $\ph_a$ in such a way that $\forall h\in J\ \ta(h)=\psi(P(\alpha_ha))$,
that is, we put
\[\forall n\in\mathbb{Z}\quad\ph_a(n)=\frac{1}{\sqrt{m}}\sum_{h\in J}\psi(P(\alpha_ha))e_m(-nh).\]
Let $\{h_1,\cdots,h_k\}$ be a non-empty subset of $J$ with at most $r$ elements.
By Vandermonde's formula, the polynomial
$P(\alpha_{h_1}X)+\cdots+P(\alpha_{h_k}X)$ has at least one non-zero coefficient
and has odd degree. Applying Weil's theorem, we obtain
\[\Bigl|\sum_{a\in\Omega_m}\ta(h_1)\cdots\ta(h_k)\Bigr|\le(2r-2)2^{r/2}\ll m^{3/2}\log m. \]
This implies Assumption 3 with any  $1<\eta<3/2$. Indeed, take a sequence
$(\X(h))_{h\in\mathbb{Z}}$ of I.I.D. random variables such that
$\pr(\X(h)=1)=\pr(\X(h)=-1)=1/2$. For every positive integer $k\le 3\log m/\log 2$ and
every $(h_1,\dots,h_k)\in J^k$, one has
\[\frac{1}{|\Omega_m|}\sum_{a\in\Omega_m}\ta(h_1)\cdots\ta(h_k)=\ex(\X(h_1)\cdots\X(h_k))+O\Bigl(\frac{\log m}{m^{3/2}}\Bigr)\]
(one can in fact prove this estimate for all positive $k\le2r$). Thus, we deduce that our family satisfies Assumption 2 with $N=1$, and Assumption 3 with $\eta=4/3$.


To conclude the proof of Proposition \ref{HugeFamily}, let us look at $a=1$: we
have $\widehat{\ph_1}(h)=1$ if $h\ge1$ and $\widehat{\ph_1}(h)=-1$ if $h\le0$.
Using the estimate \eqref{Planche2} we deduce
\[\M(\ph_1)\ge\Bigl|\sum_{0\le n\le m/2}\ph_1(n)\Bigr|=\Bigl|\sqrt{m}\sum_{1\le n\le m/2}\frac{1-(-1)^n}{i\pi n}+O(\sqrt{m})\Bigr|=\frac{\sqrt{m}}{\pi}\log m+O(\sqrt{m}).\]


\section{Applications to families of $\ell$-adic trace functions: Proof of Corollaries \ref{GBirch}, \ref{GKloo} and \ref{HypKloo} }\label{TraceSection}
In this section we recall some notions of the formalism of $\ell$-adic trace functions and list some examples of families of functions for which we can apply Theorem $\ref{Main}$ and Theorem $\ref{Lower}$. We are assuming the definition of a constructible $\ell$-adic sheaf, for which, as well as for a general introduction on the subject, we refer the reader to \cite[Section $1.1$]{Del80} and \cite[Section $3.4.2$]{Kat80}. We start by introducing the notion of a trace function attached to a constructible $\ell$-adic sheaf as in \cite[$7.3.7$]{Kat90}. \newline In the following $p,\ell >2$ are distinct prime numbers and $\iota:\overline{\mathbb{Q}}_{\ell}\hookrightarrow \mathbb{C}$ is a fixed embedding. Let $\mathcal{F}$ be a constructible $\ell$-adic sheaf on $\overline{\mathbb{A}}_{\F}^{1}$. For any $x\in\overline{\mathbb{A}}_{\F}^{1}(\Fn)$ one defines
\[
t_{\mathcal{F},n}(x):=\iota(\Tr(\Fr_{p^{n}}|\mathcal{F}_{\overline{x}})),
\]
where $\Fr_{p^{n}}$ is the geometric Frobenius automorphism of $\Fn$ and $\mathcal{F}_{\overline{x}}$ is the stalk of $\mathcal{F}$ at a geometric point $\overline{x}$ over $x$. The function $t_{\mathcal{F},n}$ is called \textit{the trace function attached to $\mathcal{F}$ over $\Fn$}. If there is not ambiguity, we denote by $t_{\mathcal{F}}$ the trace function $t_{\mathcal{F},1}$. 
\begin{defin}
Let $\mathcal{F}$ be a constructible $\ell$-adic sheaf on $\overline{\mathbb{A}}_{\F}^{1}$ and $j:U\hookrightarrow\overline{\mathbb{A}}_{\F}^{1}$ the largest dense open subset of $\overline{\mathbb{A}}_{\F}^{1}$ where $\mathcal{F}$ is lisse.
\begin{itemize}
\item[$i)$] The sheaf $\mathcal{F}$ is said to be a \textit{middle-extension} $\ell$-adic sheaf if $\mathcal{F}=j_{*}j^{*}\mathcal{F}$ (see \cite[$4.4$, $4.5$]{Kat80} for the definition of $j_{*}j^{*}\mathcal{F}$ and its basic properties).
\item[$ii)$] The sheaf $\mathcal{F}$ is said to be a \textit{middle-extension $\ell$-adic sheaf, punctually pure of weight $0$} if it is a middle-extension sheaf and if for every $n\geq 1$ and every $x\in U(\Fn)$, the images of the eigenvalues of $(\Fr_{p^{n}}|\mathcal{F}_{\overline{x}})$ via the fixed embedding $\iota:\overline{\mathbb{Q}}_{\ell}\hookrightarrow \mathbb{C}$, are complex numbers of modulus $1$.
\end{itemize}
\end{defin}

\begin{rem} 
 In accordance with \cite[Definition $1.12$]{FKM15}, we will refer to trace functions attached to punctually pure of weight $0$ middle-extension $\ell$-adic sheaves on $\overline{\mathbb{A}}_{\F}^{1}$ as \textit{trace functions}.
\end{rem}

\subsection{Conductor and Fourier transform}
\begin{defin}[\cite{FKM19}, pp. $4-6$] Let $\mathcal{F}$ be a constructible $\ell$-adic sheaf on $\overline{\mathbb{A}}_{\F}^{1}$ and $j:U\hookrightarrow\overline{\mathbb{A}}_{\F}^{1}$ the largest dense open subset of $\overline{\mathbb{A}}_{\F}^{1}$ where $\mathcal{F}$ is lisse. The \textit{conductor of $\mathcal{F}$} is defined as
\[
c(\mathcal{F}):=\rank(\mathcal{F})+|\sing(\mathcal{F})|+\sum_{x\in \overline{{\mathbb{P}}}_{\F}^{1}(\overline{\mathbb{F}}_{p})}\text{Swan}_{x}(j_{*}j^{*}\mathcal{F})+\dim H_{c}^{0}(\overline{\mathbb{A}}_{\F}^1,\mathcal{F}),
\]
where
\begin{enumerate}
\item[$i)$] $\rank(\mathcal{F}):=\dim \mathcal{F}_{x}$, for any $x$ where $\mathcal{F}$ is lisse.
\item[$ii)$] $\sing(\mathcal{F}):=\{x\in\overline{\mathbb{P}}_{\F}^{1}(\overline{\mathbb{F}}_{p}):\mathcal{F} \text{ is not lisse at }x\}$.
\item[$iii)$] For any $x\in\overline{\mathbb{P}}_{\F}^{1}(\overline{\mathbb{F}}_{p})$, $\Swan_{x}(j_{*}j^{*}\mathcal{F})$ is the \textit{Swan conductor of $\mathcal{F}$ at $x$} (see \cite[Chapter $1$]{Kat88} for the definition of the Swan conductor).
\end{enumerate}
\label{defn : cond}
\end{defin}

\begin{rem}
Recall that if $\mathcal{F}$ is a middle-extension sheaf then
\[
c(\mathcal{F})=\rank(\mathcal{F})+|\sing(\mathcal{F})|+\sum_{x\in \overline{{\mathbb{P}}}_{\F}^{1}(\overline{\mathbb{F}}_{p})}\text{Swan}_{x}(\mathcal{F}),
\]
since in this case $\mathcal{F}\cong j_{*}j^{*}\mathcal{F}$ and $\dim H_{c}^{0}(\overline{\mathbb{A}}_{\F}^1,\mathcal{F})=0$.
\end{rem}

Let $\mathcal{F}$ be a middle-extension $\ell$-adic sheaf on $\overline{\mathbb{A}}_{\F}^{1}$ and let $t_{\mathcal{F},n}$ be the trace function attached to $\mathcal{F}$ over $\Fn$. We recall that the normalized Fourier transform of $t_{\mathcal{F},n}$ is given by
\[
\FT(t_{\mathcal{F},n})(x):=-\frac{1}{\sqrt{p^{n}}}\sum_{y\in\mathbb{F}_{p^{n}}}t_{\mathcal{F},n}(y)e_{p}(\Tr_{\mathbb{F}_{p^{n}}/\mathbb{F}_{p}}(xy)).
\]
In what follows, we recall some key results about the Fourier transform in the trace functions setting.
\begin{defin}[\cite{Kat88}, Definition $8.2.1.2$]
A middle-extension $\ell$-adic sheaf $\mathcal{F}$ over $\overline{\mathbb{A}}_{\mathbb{F}_{p}}^1$ is said to be a \textit{Fourier sheaf} if it does not contain any Artin-Schreier sheaf $\mathcal{L}_{e_{p}(aT)}$ in its geometric Jordan-H\"older decomposition.
\end{defin}

\begin{thm}
Let $\mathcal{F}$ be a middle-extension $\ell$-adic Fourier sheaf over $\overline{\mathbb{A}}_{\mathbb{F}_{p}}^1$. Then there exists a middle-extension $\ell$-adic sheaf $\FT(\mathcal{F})$ over $\overline{\mathbb{A}}_{\mathbb{F}_{p}}^1$ such that
\[
t_{\FT(\mathcal{F}),n}(x)=-\frac{1}{p^{n/2}}\sum_{y\in\mathbb{F}_{p^{n}}}t_{\mathcal{F},n}(y)e_{p}(\Tr_{\mathbb{F}_{p^{n}}/\mathbb{F}_{p}}(xy)),
\]
for any $n\geq 1$ and any $x\in\mathbb{F}_{p^{n}}$. Moreover one has that:
\begin{itemize}
\item[$i)$] If $\mathcal{F}$ is geometrically irreducible,  punctually pure of weight $0$ then the same holds for $\FT(\mathcal{F})$. Moreover $\FT(\mathcal{F})$ is a Fourier sheaf with the property
\[
\FT(\FT(\mathcal{F}))=[\times(-1)]^{*}\mathcal{F},
\]
where $[\times(-1)]^{*}\mathcal{F}$ denotes the inverse image sheaf $[x\mapsto -x]^{*}\mathcal{F}$. 
\item[$ii)$]$\rank(\FT(\mathcal{F}))\leq \sum_{x}\Swan_{x}(\mathcal{F})+|\sing(\mathcal{F})|\cdot\rank (\mathcal{F})$,
\item[$iii)$] $ |\sing(\FT(\mathcal{F}))|\leq 2+\rank (\mathcal{F})$,
\item[$iv)$]$\Swan_{x}(\FT(\mathcal{F}))\leq c(\mathcal{F})$.
\end{itemize}
In particular, $c(\FT(\mathcal{F}))\leq 10c(\mathcal{F})^{2}$.
\label{thm : Fouriersheaf}
\end{thm}
\begin{proof}
The construction of the sheaf $\FT(\mathcal{F})$ and the proof of $(i)$ can be found in \cite[Definition $8.2.3$]{Kat88}, \cite[Theorem $8.2.5$]{Kat88} and \cite[Theorem $8.4.1$]{Kat88}. For the part $(ii),(iii)$ and $(iv)$, we refer to the proof of \cite[Proposition $8.2$]{FKM15}.
\end{proof}

Here the main examples of trace functions we should keep in mind (many of them already appearing in Paragraph $\ref{sec : examp}$):
\begin{itemize}
\item[$(i)$] For any $f,g\in\mathbb{F}_{p}(T)$, and any multiplicative character $\chi$ on $\Fs$, the function $x\mapsto e(f(x)/p)\chi (g(x))$ is the trace function attached to the Artin-Schreier sheaf $\mathcal{L}_{e_{p}(f(T))\chi (g(T))}$.
\item[$(ii)$] The $r$-th hyper-Kloosterman sums: the map
\[
x\mapsto\Kl_{r}(x;p)=\frac{(-1)^{r-1}}{p^{(r-1)/2}}\sum_{\substack{y_{1},...,y_{r}\in\mathbb{F}_{p}^{\times}\\y_{1}\cdot...\cdot y_{r}=x}}e\Big(\frac{y_{1}+\cdots +y_{r}}{p}\Big)
\]
can be seen as the trace function attached to the \textit{Kloosterman sheaf} $\mathcal{K}\ell_{r}$ (see \cite{Kat88} for the definition of such sheaf and for its basic properties).
\item[$iii)$] The Birch sums, $a\mapsto -\text{Bi}_{p}(a)$, can be seen as the trace function attached to the sheaf $\FT (\mathcal{L}_{e_{p}(T^{3})})$.
\end{itemize}
\subsection{$\lambda$-parameter families}
\begin{defin}
Let $r\geq 2$ be an integer. A middle-extension $\ell$-adic sheaf $\mathcal{K}$ is said of $\Sp_{2r}$-type if
\begin{itemize}
\item[$i)$] $\mathcal{K}$ is  punctually pure of weight $0$,
\item[$ii)$] one has $G_{\mathcal{K}}^{\text{arith}}=G_{\mathcal{K}}^{\text{geom}}=\Sp_{2r}(\mathbb{C})$ (see \cite[Chapter $3$]{Kat88} for the definition of the monodromy groups).
\end{itemize}
\end{defin}
We are finally ready to introduce the notion of a \textit{$\lambda$-parameter family} which is a refinement of \cite[Definition $2.6$]{Per17}:
\begin{defin}
Let $r\geq 2$ and $\lambda\geq 1$ be integers. A family $\{\mathcal{F}_{\mathbf{a}}\}_{\mathbf{a}\in(\Fs)^{\lambda}}$ is said to be a \textit{$\lambda$-parameter family of $\Sp_{2r}$-type} if the following conditions hold
\begin{itemize}
\item[$i)$] for any $\mathbf{a}\in(\Fs)^{\lambda}$, $\mathcal{F}_{\mathbf{a}}$ is a Fourier, irreducible middle-extension $\ell$-adic sheaf on $\overline{\mathbb{A}}_{\mathbb{F}_{p}}^{1}$  punctually pure of weight $0$,
\item[$ii)$] there exists $C\geq1$ such that
\[
c(\mathcal{F}_{\mathbf{a}})\leq C
\] 
for any $\mathbf{a}\in(\Fs)^{\lambda}$. We call the smallest $C$ with this property the \textit{conductor of the family} and we denote it by $C_{\mathfrak{F}}$,
\item[$iii)$] for any $\mathbf{a}'\in(\Fs)^{\lambda-1}$ the $\ell$-adic sheaf $\mathcal{K}_{\mathbf{a}',1}=\text{FT}(\mathcal{F}_{\mathbf{a}',1})$ is of $\Sp_{2r}$-type,
\item[$iv)$] for any $a,z\in\Fs$ and $\mathbf{a}'\in(\Fs)^{\lambda-1}$, there exists $\gamma_{z}\in\PGL_{2}(\F)$ such that 
\[
t_{\mathcal{K}_{\mathbf{a}',a}}(z)=t_{\mathcal{K}_{\mathbf{a}',1}}(\gamma_{z}(a)),
\]
and $\gamma_{z}\neq\gamma_{z'}$ if $z\neq z'$,
\item[$v)$] for any $\mathbf{a}'\in(\Fs)^{\lambda-1}$ and any $z_{1},z_{2}\in \Fs$, if $z_{1}\neq z_{2}$ then
\[
[\gamma_{z_{1}}\gamma_{z_{2}}^{-1}]^{*}\mathcal{K}_{\mathbf{a}',1}\neq_{\text{geom}}\mathcal{K}_{\mathbf{a}',1}\otimes\mathcal{L}
\]
for any $\ell$-adic sheaf $\mathcal{L}$ of rank $1$,
\item[$vi)$] there exists $\delta,\alpha >0$ such that for any interval $I$ of length $|I|\leq p^{1/2+\delta}$, one has
\[
\frac{1}{p^{\lambda}}\sum_{\mathbf{a}\in (\Fs)^{\lambda}}\Big|\frac{1}{\sqrt{p}}\sum_{n\in I}t_{\mathbf{a}}(n)\Big|^{\alpha}\ll p^{-1/2-\delta}.
\]
\end{itemize}
\label{def : 2par}
\end{defin}

\begin{Pro}
Let $\lambda,r\geq 1$ be integers. Let $\{\mathcal{F}_{\mathbf{a}}\}_{\mathbf{a}\in(\Fs)^{\lambda}}$ be a $\lambda$-parameter family of $\Sp_{2r}$-type. Then  Theorem $\ref{Main}$ and $\ref{Lower}$ hold for $\{t_{\mathcal{F}_{\mathbf{a}}}\}_{\mathbf{a}\in(\Fs)^{\lambda}}$. 
\end{Pro}
\begin{proof}
It is enough to show that the set $\{t_{\mathcal{F}_{\mathbf{a}}}\}_{\mathbf{a}\in(\Fs)^{\lambda}}$ satisfies Assumptions 1, 2, 3 and 4. We start with Assumption 1: for any $\mathbf{a}\in(\Fs)^{\lambda}$ one has that $\|t_{\mathcal{F}_{\mathbf{a}}}\|_{\infty}\leq\rank (\mathcal{F}_{\mathbf{a}})$ thanks to \cite[Lemma $1.8.1$]{Del80}. Hence, using $(ii)$ in the definition of a $\lambda$-parameter family we get $\|t_{\mathcal{F}_{\mathbf{a}}}\|_{\infty}\leq c (\mathcal{F}_{\mathbf{a}})\leq C_{\mathfrak{F}}$. Assumption 2 is just an application of \cite[Lemma $3.2$]{Per17}. Let us check Assumption 3. Let $(h_{1},...,h_{k})\in(-p/2,p/2]^{k}$ with $h_{i}\neq 0$ for $i=1,...,k$ and consider
\begin{equation}
\frac{1}{(p-1)^{\lambda}}\sum_{\mathbf{a}'\in(\Fs)^{\lambda-1}}\sum_{a\in\Fs}t_{\mathcal{K}_{\mathbf{a}',a}}(h_{1})\cdots t_{\mathcal{K}_{\mathbf{a}',a}}(h_{k}).
\label{eq : assum2}
\end{equation}
We know that for any $\mathbf{a}'\in(\Fs)^{\lambda-1}$ and any $a,h_{i}\in \Fs$, it holds that $t_{\mathcal{K}_{\mathbf{a}',a}}(h_{i})=t_{\mathcal{K}_{\mathbf{a}',1}}(\gamma_{h_{i}}(a))$ for some $\gamma_{h_{i}}\in\PGL_{2}(\F)$ ($(iv)$ in Definition $\ref{def : 2par}$). Thus we can rewrite  equation $(\ref{eq : assum2})$ as
\[
\frac{1}{(p-1)^{\lambda}}\sum_{\mathbf{a}\in(\Fs)^{\lambda-1}}\sum_{a\in\Fs}t_{\mathcal{K}_{\mathbf{a}',1}}(\gamma_{h_{1}}(a))\cdots t_{\mathcal{K}_{\mathbf{a}',1}}(\gamma_{h_{k}}(a)).
\]
Thanks to the property $(v)$ in Definition $\ref{def : 2par}$, we can argue as in \cite[$4.2.1$]{Per17}  to obtain
\[
(-1)^k\sum_{a\in\Fs}t_{\mathcal{K}_{\mathbf{a}',1}}(\gamma_{h_{1}}(a))\cdots t_{\mathcal{K}_{\mathbf{a}',1}}(\gamma_{h_{k}}(a))=\ex(\X (h_{1})...\X (h_{k}))(p-1)+O(c(\mathcal{H})\sqrt{p}),
\]
where $\mathcal{H}=[\gamma_{h_{1}}]^{*}\mathcal{K}_{\mathbf{a}',1}\otimes\cdots\otimes [\gamma_{h_{k}}]^{*}\mathcal{K}_{\mathbf{a}',1}$ and the $\X (h_{i})$'s are independent random variables uniformly distributed with respect to the Haar measure on $\USp_{2r}$ which satisfy Assumptions 3a (Lemma $\ref{traceUSp}$) and 3b. Let us bound $c(\mathcal{H})$. Recall that
\[
c(\mathcal{H})=\rank(\mathcal{H})+|\sing(\mathcal{H})|+\sum_{x}\text{Swan}_{x}(\mathcal{H}).
\]
One has that $\rank (\mathcal{H})=\prod_{i}\rank ([\gamma_{h_{i}}]^{*}\mathcal{K}_{\mathbf{a}',1})=\rank (\mathcal{K}_{\mathbf{a}',1})^{k}$ and that $|\sing (\mathcal{H})|\leq\sum_{i}|\sing ([\gamma_{h_{i}}]^{*}\mathcal{K}_{\mathbf{a}',1})|\leq k|\sing (\mathcal{K}_{\mathbf{a}',1})|$. On the other hand, \cite[Lemma $1.3$]{Kat88} implies that
\[
\Swan_{x}(\mathcal{H})\leq\rank (\mathcal{H})\cdot\Big(\sum_{i=1}^{k}\Swan_{x}([\gamma_{h_{i}}]^{*}\mathcal{K}_{\mathbf{a}',1})\Big)\leq \rank (\mathcal{K}_{\mathbf{a}',1})^{k}kc(\mathcal{K}_{\mathbf{a}',1}).
\]
Thus we have that
\[
\begin{split}
c(\mathcal{H})&\leq \rank (\mathcal{K}_{\mathbf{a}',1})^{k}+ k|\sing (\mathcal{K}_{\mathbf{a}',1})|+ k|\sing (\mathcal{K}_{\mathbf{a}',1})|\cdot\rank (\mathcal{K}_{\mathbf{a}',1})^{k}kc(\mathcal{K}_{\mathbf{a}',1})\\&\leq c(\mathcal{K}_{\mathbf{a}',1})^{k}+kc(\mathcal{K}_{\mathbf{a}',1})+k^2c(\mathcal{K}_{\mathbf{a}',1})^{k+2}\\&\ll C_{\mathfrak{F}}^{8k},
\end{split}
\]
where in the last step we used property $(ii)$ in Definition \ref{def : 2par} together with Theorem $\ref{thm : Fouriersheaf}$. Thus $(\ref{eq : assum2})$ becomes
\[
\frac{(-1)^k}{(p-1)^{\lambda}}\sum_{\mathbf{a}\in(\Fs)^{\lambda-1}}\sum_{a\in\Fs}t_{\mathcal{K}_{\mathbf{a}',a}}(h_{1})\cdots t_{\mathcal{K}_{\mathbf{a}',a}}(h_{k})=\ex\left(\X(h_1) \dots \X(h_k)\right)+ O\left(\frac{C_{\mathfrak{F}}^{8k}}{\sqrt{p}}\right),
\]
as we wanted. Finally, Assumption 4 simply follows from the definition of a $\lambda$-parameter family (property $(vi)$).

\end{proof}

\subsection{Examples of $1$-parameter families}
Let $g\in \mathbb{Z}[t]$ be an odd polynomial of degree $2r+1$, such that $r\ge 1$. For $p$ large enough
\[
\{\mathcal{L}_{e_{p}(\beta T+g(T))}\}_{\beta\in\Fs}
\]
is a $1$-parameter family of $\Sp_{2r}$-type.
\begin{itemize}
\item[$i)$] For any $\beta\in\Fs$, the Artin-Schreier sheaf $\mathcal{L}_{e_{p}(\beta T+g(T))}$ is a Fourier, irreducible middle-extension $\ell$-adic sheaf on $\overline{\mathbb{A}}_{\mathbb{F}_{p}}^{1}$ punctually pure of weight $0$. Moreover its trace function is $t_{\mathcal{F}_{\beta}}: x\mapsto e_{p}(\beta x+g(x))$.
\item[$ii)$] One has that $\sing (\mathcal{L}_{e_{p}(\beta T+g(T))})=\{\infty\}$ for any $\beta\in\Fs$. Moreover, if $p>2r+1$ then $\Swan_{\infty}(\mathcal{L}_{e_{p}(\beta T+g(T))})=\deg g=2r+1$. Thus, $c(\mathcal{L}_{e_{p}(\beta T+g(T))})= 2r+3$ for any $\beta \in\Fs$.
\item[$iii)$] In \cite[$7.13$ $\Sp$-example $(2)$]{Kat90}, it is shown that the sheaf $\mathcal{K}_{1}$ is such that $G_{\mathcal{K}_{1}}^{\text{geom}}=\Sp_{2r}(\mathbb{C})$. Moreover, as explained in \cite[7.2.(1)]{Per17}, this implies that $G_{\mathcal{K}_{1}}^{\text{geom}}=G_{\mathcal{K}_{1}}^{\text{arith}}$.
\item[$iv)$] Let $\beta\in\Fs$. By definition of the Fourier transform we have that
\[
\begin{split}
t_{\mathcal{K}_{\beta}}(z)&=-\frac{1}{\sqrt{p}}\sum_{x\in\F}e_{p}(x+ g(x)+(\beta+z-1)x)\\&=t_{\mathcal{K}_{1}}(\beta+z-1)
\end{split}
\]
and therefore $t_{\mathcal{K}_{\beta}}(z)=t_{\mathcal{K}_{1}}(\gamma_{z}(\beta))$ for $\gamma_{z}:=\begin{pmatrix}1 & z-1\\ 0 & 1\end{pmatrix}$.
\item[$v)$] This is done in \cite[Proposition $7.5$]{Per17}.
\item[$vi)$] By Weyl's method (see for example \cite[Lemma $20.3$]{IK}), there exists $\eta>0$ such that
\[
\Big|\frac{1}{\sqrt{p}}\sum_{n\in I}e_{p}(n\beta +g(n))\Big|\ll p^{-\eta}
\]
for any interval $I$ of length $|I|\leq p^{1/2+\eta}$. Moreover, $\eta$ and the implied constant depend only on $\deg g$. Thus, for any $\alpha$,
\[
\frac{1}{p-1}\sum_{\beta\in \Fs}\Big|\frac{1}{\sqrt{p}}\sum_{n\in I}e_{p}(n\beta +g(n))\Big|^{\alpha}\ll p^{-\alpha\eta}.
\] 
Choosing a suitable $\alpha >1$, property $(vi)$ in Definition \ref{def : 2par} is satisfied.
\end{itemize}

\subsection{Examples of $2$-parameter families}
In this section we present some examples of families of $\Sp_{2r}$-type. In the following, for a sheaf $\mathcal{F}_{a,b}$ we denote $\mathcal{K}_{a,b}=\FT (\mathcal{F}_{a,b})$.
\subsubsection{Exponential sums II}
\label{ex : 1}
Let $d\in\mathbb{N}_{\geq 1}$ with $d$ odd. For $p$ large enough
\[
\{\mathcal{L}_{e_{p}(bT+(a\overline{T})^{d}))}\}_{(a,b)\in\Fs\times\Fs}
\]
is a $2$-parameter family of $\Sp_{d+1}$-type.
\begin{itemize}
\item[$i)$] For any $(a,b)\in\Fs\times\Fs$, the Artin-Schreier sheaf $\mathcal{L}_{e_{p}(bT+(a\overline{T})^{d})}$ is a Fourier, irreducible middle-extension $\ell$-adic sheaf on $\overline{\mathbb{A}}_{\mathbb{F}_{p}}^{1}$ punctually pure of weight $0$. Moreover, its trace function is $t_{\mathcal{F}_{a,b}}: x\mapsto e_{p}(bx+(a\overline{x})^{d})$.
\item[$ii)$] One has that $\sing (\mathcal{L}_{e_{p}(bT+(a\overline{T})^{d})})=\{0,\infty\}$ for any $(a,b)\in\Fs\times\Fs$. Moreover, if $d<p$ then
\[
\Swan_{0}(\mathcal{L}_{e_{p}(bT+(a\overline{T})^{d})})=d,\quad \Swan_{\infty}(\mathcal{L}_{e_{p}(bT+(a\overline{T})^{d})})=1.
\]
Thus $c(\mathcal{L}_{e_{p}(bT+(a\overline{T})^{d})})= d+4$ for any $(a,b)\in\Fs\times \Fs$.
\item[$iii)$] In \cite[Theorem $7.12.3.1$]{Kat90} is it shown that for any $a\in\Fs$,
\[
G_{\mathcal{K}_{a,1}}^{\text{geom}}=
\begin{cases}
\Sp_{d+1}(\mathbb{C})		&\text{if $d$ is odd}\\
\SL_{d+1}(\mathbb{C})		&\text{if $d$ is even}.
\end{cases}
\]
Hence, when $d$ is odd one concludes also that $G_{\mathcal{K}_{a,1}}^{\text{geom}}=G_{\mathcal{K}_{a,1}}^{\text{arith}}$ (by \cite[$7.2.(1)$]{Per17}).
\item[$iv)$] Let $(a,b)\in\Fs\times\Fs$. By definition of the Fourier transform we have
\[
\begin{split}
t_{\mathcal{K}_{a,b}}(z)&=-\frac{1}{\sqrt{p}}\sum_{x\in\mathbb{F}_{p}^{\times}}e_{p}(bx+(a\overline{x})^{d}+xz)\\&=-\frac{1}{\sqrt{p}}\sum_{x\in\mathbb{F}_{p}^{\times}}e_{p}((a\overline{x})^{d}+(b+z)x)\\&=t_{\mathcal{K}_{a,1}}(b+z-1)
\end{split}
\]
and thus $t_{\mathcal{K}_{a,b}}(z)=t_{\mathcal{K}_{a,1}}(\gamma_{z}(b))$ for $\gamma_{z}=\begin{pmatrix}1 & z-1\\ 0 & 1\end{pmatrix}$.
\item[$v)$] For any $a,z_{1},z_{2}\neq 0$ with $z_{1}\neq z_{2}$, we need to prove that
\begin{equation}
[\gamma_{z_{1}}\gamma_{z_{2}}^{-1}]^{*}\mathcal{K}_{a,1}\neq_{\text{geom}}\mathcal{K}_{a,1}\otimes\mathcal{L},
\label{eq : pro(iv)}
\end{equation}
for any $\ell$-adic sheaf $\mathcal{L}$ of rank $1$, where $\gamma_{z_{1}},\gamma_{z_{2}}$ are as in $(iv)$. First of all, observe that $\gamma_{z_{1}}\gamma_{z_{2}}^{-1}=\begin{pmatrix}1 & z_{1}-z_{2}\\ 0 & 1\end{pmatrix}$. Let us denote $\mathcal{K}_{a,0}=\FT (\mathcal{L}_{e_{p}((a\overline{T})^{d})})$. Arguing as in $(iv)$, one has that $\mathcal{K}_{a,1}=_{\text{geom}}[\tau]^{*}\mathcal{K}_{a,0}$ with $\tau=\begin{pmatrix}1 & 1\\ 0 & 1\end{pmatrix}$. Thus, in order to check $(\ref{eq : pro(iv)})$ it is enough to show that
\[
[\gamma_{z}']^{*}\mathcal{K}_{a,0}\neq_{\text{geom}}\mathcal{K}_{a,0}\otimes\mathcal{L},
\]
for any $\ell$-adic sheaf $\mathcal{L}$ of rank $1$, where $\gamma_{z}'=\begin{pmatrix}1 & z\\ 0 & 1\end{pmatrix}$ with $z\neq 0$. Since $\mathcal{L}_{e_{p}((a\overline{T})^{d})}$ is lisse at $\{\infty\}$, then $\mathcal{K}_{a,0}=\FT(\mathcal{L}_{e_{p}((a\overline{T})^{d})})$ is singular at $\{0,\infty\}$ (\cite[Corollary $8.5.8$]{Kat88}). Moreover, using Laumon's theory of the $\ell$-adic Fourier transform, one gets that $\mathcal{K}_{a,0}$ has an unique slope at $d/(d+1)$ at $\infty$, and that it is tame at $0$ (\cite[Theorem $7.5.4$]{Kat90}). Thus,
\[
\Swan_{0}(\mathcal{K}_{a,0})=0,\qquad \Swan_{\infty}(\mathcal{K}_{a,0})=d.
\]
Then for any $a,z$ one has $\sing ([\gamma_{z}']^{*}\mathcal{K}_{a,0})=\{-z ,\infty\}$. Moreover, $[\gamma_{z}']^{*}\mathcal{K}_{a,0}$ has an unique slope at $d/(d+1)$ at $\infty$ and it is tame at $-z$. By contradiction, assume that there exists a rank $1$ sheaf $\mathcal{L}$ such that $[\gamma_{z}']^{*}\mathcal{K}_{a,0}=_{\text{geom}}\mathcal{K}_{a,0}\otimes\mathcal{L}$.  Without loss of generality, we may assume that $\mathcal{L}$ is punctually pure of weight $0$ since $\mathcal{K}_{a,0}$ is punctually pure of weight $0$. From the discussion above it would follow that $\{0,-z\}\subset \sing (\mathcal{L})\subset\{0,-z,\infty\}$, and that $\mathcal{L}$ is tame everywhere. At this point it is useful to compute some data about $\FT(\mathcal{L})$:
\begin{itemize}
\item[$a)$] $\sing (\FT (\mathcal{L}))=\{0,\infty\}$, since $\mathcal{L}$ is tame at $\infty$ (\cite[Corollary $8.5.8$]{Kat88}),
\item[$b)$] $\rank (\FT (\mathcal{L}))=\dim(H_{c}^{1}(\overline{\mathbb{A}}_{\mathbb{F}_{p}}^{1},\mathcal{L}\otimes\mathcal{L}_{e_{p}(\alpha T)}))$ for any $\alpha\neq 0$. On the other hand, the Grothendieck-Ogg-Shafarevich formula  (\cite[Chapter $14$]{Kat12}) implies that
\[
\begin{split}
\dim(H_{c}^{1}(\overline{\mathbb{A}}_{\mathbb{F}_{p}}^{1},\mathcal{L}\otimes\mathcal{L}_{e_{p}(\alpha T)}))&=-\rank (\mathcal{L}\otimes\mathcal{L}_{e_{p}(\alpha T)})\\&+\sum_{x\in\overline{\mathbb{A}}_{\mathbb{F}_{p}}^{1}}\Drop_{x}(\mathcal{L}\otimes\mathcal{L}_{e_{p}(\alpha T)})\\&+\sum_{x\in\overline{\mathbb{P}}_{\mathbb{F}_{p}}^{1}}\Swan_{x}(\mathcal{L}\otimes\mathcal{L}_{e_{p}(\alpha T)})\\&=-1+2+1=2,
\end{split}
\]
where for a constructible $\ell$-adic sheaf $\mathcal{F}$, $\Drop_{x}(\mathcal{F})=\rank\mathcal{F}-\dim\mathcal{F}_{x}$. Hence, $\rank (\FT (\mathcal{L}))=2$.
\item[$c)$] Since $\mathcal{L}$ is not lisse on $\overline{\mathbb{A}}_{\F}^{1}\setminus\{0\}$, \cite[Corollary $8.5.8$]{Kat88} implies that $\FT (\mathcal{L})(\infty)$ has a break at $1$. Thus $\Swan_{\infty}(\FT (\mathcal{L}))\geq 1$. 
\end{itemize}
Since we are assuming that $[\gamma_{z}']^{*}\mathcal{K}_{a,0}=_{\text{geom}}\mathcal{K}_{a,0}\otimes\mathcal{L}$, it follows that 
\[\FT (\mathcal{K}_{a,0}\otimes\mathcal{L})=_{\text{geom}}\FT ([\gamma_{z}']^{*}\mathcal{K}_{a,0})=\mathcal{L}_{e_{p}(-zT-(a\overline{T})^d)}.
\]
Hence, for any $n\geq 1$ there exists a complex number $\omega_{n}$ of modulus $1$ such that for any $s\in\Fs$ one has
\begin{equation}
\begin{split}
e_{p}(\Tr_{\Fn/\F}(-zs-(a\overline{s})^d))&= \FT (t_{[\gamma_{z}']^{*}\mathcal{K}_{a,0},n})(s)\\&=\omega_{n} \cdot\FT (t_{\mathcal{K}_{a,0},n}\cdot t_{\mathcal{L},n})(s)\\&=-\frac{\omega_{n}}{p^{n/2}}\sum_{y}\FT (t_{\mathcal{L},n})(y)\FT(t_{\mathcal{K}_{a,0},n})(s-y)\\&=-\frac{\omega_{n}}{p^{n/2}}\sum_{y}\FT (t_{\mathcal{L},n})(y)e_{p}(\Tr_{\Fn/\F}((a(\overline{y-s}))^{d})).
\end{split}
\label{eq : fourid}
\end{equation}

The next step is to show that we can see $\FT (\mathcal{K}_{a,0}\otimes\mathcal{L})$ as a cohomological transform in the sense of \cite{FKM19}. Let us start with the Artin-Schreier sheaf on $\overline{\mathbb{A}}_{\F}^{1}\times\overline{\mathbb{A}}_{\F}^{1}$, $\mathcal{L}_{e_{p}((a(\overline{Y-T}))^{d})}$ (see \cite[Definition $2.2$]{FKM19}), and consider the $\ell$-adic sheaf $\mathcal{G}:=R^{1}p_{1,!}(p_{2}^{*}\FT(\mathcal{L})\otimes\mathcal{L}_{e_{p}((a(\overline{Y-T}))^{d})})(1/2)$, where $p_{1},p_{2}$ are the two projections $p_{i}:\overline{\mathbb{A}}_{\F}^{1}\times\overline{\mathbb{A}}_{\F}^{1}\rightarrow \overline{\mathbb{A}}_{\F}^{1}$.
As pointed out in \cite[Remark $2.4.(1)$]{FKM19}, $\mathcal{G}$ is a constructible $\ell$-adic sheaf. Let $j:U\hookrightarrow\overline{\mathbb{A}}^{1}_{\F}$ be the largest dense open subset on which $\mathcal{G}$ is lisse (we remark that $\sing (\mathcal{G})\leq c(\mathcal{G})\ll_{d}1$, thanks to \cite[Theorem $2.3$]{FKM19}), and consider the sheaf $\mathcal{H}=j_{*}j^{*}\mathcal{G}$; this is a middle-extension $\ell$-adic sheaf punctually pure of weight $0$ on $U$, such that $\rank \mathcal{H}=\rank{\mathcal{G}}$ and for any $n\geq 1$ and any $s\in U(\Fn)$
\[
t_{\mathcal{H},n}(s)=t_{\mathcal{G},n}(s)=-\frac{1}{p^{n/2}}\sum_{y}\FT (t_{\mathcal{L},n})(y)e_{p}(\Tr_{\Fn/\F}((a(\overline{y-s}))^{d})),
\]  
thanks to \cite[Corollary $5.3$]{FKM19}. Moreover, one has that
\[
\begin{split}
\frac{1}{p^n}\sum_{s\in\Fn}|t_{\mathcal{H},n}(s)|^{2}&=\frac{1}{p^n}\sum_{s\in U(\Fn)}|t_{\mathcal{H},n}(s)|^{2}+O_{d}(1/p^n)\\&=\frac{1}{p^n}\sum_{s\in U(\Fn)}|e_{p}(\Tr_{\Fn/\F}(-sx-(a\overline{s})^d))|^{2}+O_{d}(1/p^n)\\&=1+O_{d}(1/p^n).
\end{split}
\]
Hence, using \cite[Lemma $7.0.3$]{Kat96}, we conclude that $\mathcal{H}$ is geometrically irreducible.\newline
 Since the trace functions attached to the irreducible middle-extension sheaves $\mathcal{L}_{e_{p}(-zT-(a\overline{T})^d)}$ and $\mathcal{H}$ coincide up to a multiplicative factor of modulus $1$ on $U(\F)$ and $c(\mathcal{H}), c(\mathcal{L}_{e_{p}(-zT-(a\overline{T})^d)})\ll_{d} 1$ (\cite[Theorem $2.3$]{FKM19}), it follows that $\mathcal{L}_{e_{p}(-zT-(a\overline{T})^d)}=_{geom}\mathcal{H}$ thanks to \cite[Corollary 3.6]{FKM13}. In particular, $\rank \mathcal{H}=\rank (\mathcal{L}_{e_{p}(-zT-(a\overline{T})^d)})$. To get a contradiction it is enough to show that $\rank\mathcal{H}=\rank\mathcal{G}\geq d+1$: in this case we would get that $1=\rank (\mathcal{L}_{e_{p}(-zT-(a\overline{T})^d)})=\rank{\mathcal{G}}\geq d+1>1$ which is absurd. We know that $\rank\mathcal{G}$ is equal to the dimension of the stalk $\mathcal{G}_{s}$ for any $s\in\mathbb{A}_{\mathbb{F}_{p}}^{1}$, where $\mathcal{G}$ is lisse. Using the Proper Base-Change Theorem (\cite[Arcata, IV, Theorem $5.4$]{Del77}) one gets
\[
\begin{split}
\mathcal{G}_{s}&=(R^{1}p_{1,!}(p_{2}^{*}\FT(\mathcal{L})\otimes\mathcal{L}_{e_{p}((a(\overline{Y-T}))^{d})}))_{s}\\&=H_{c}^{1}(\overline{\mathbb{A}}_{\mathbb{F}_{p}}^{1},\FT (\mathcal{L})\otimes\mathcal{L}_{e_{p}((a(\overline{Y-s}))^{d})}).
\end{split}
\]
Thus we need to compute
\[
N:=\dim (H_{c}^{1}(\overline{\mathbb{A}}_{\mathbb{F}_{p}}^{1},\FT (\mathcal{L})\otimes\mathcal{L}_{e_{p}((a(\overline{Y-s}))^{d})})),
\] 
for some $s\notin\sing (\mathcal{G})$. To simplify the notation let us denote $\mathcal{N}:=\FT (\mathcal{L})\otimes\mathcal{L}_{e_{p}((a(\overline{Y-s}))^{d})}$ where $s\in\overline{\mathbb{A}}_{\mathbb{F}_{p}}^{1}\setminus\{0\}$.\newline Observe that $\rank (\mathcal{N})=\rank (\FT (\mathcal{L}))\cdot\rank (\mathcal{L}_{e_{p}((a(\overline{Y-s}))^{d})})=2$. Moreover, since $\sing (\FT (\mathcal{L}))=\{0,\infty\}$ and $\sing(\mathcal{L}_{e_{p}((a(\overline{Y-s}))^{d})})=\{s\}$ we have that $\sing (\mathcal{N})=\{0,s,\infty\}$ and
\[
\begin{split}
\Swan_{0}(\mathcal{N})=\Swan_{0}(&\FT (\mathcal{L}))=0,\quad \Swan_{\infty}(\mathcal{N})=\Swan_{\infty}(\FT (\mathcal{L}))\geq 1, \\&
\Swan_{s}(\mathcal{N})=\Swan_{s}(\mathcal{L}_{e_{p}((a(\overline{Y-s}))^{d})})=d.
\end{split}
\]
Thus, using the Grothendieck-Ogg-Shafarevich formula we get
\[
\begin{split}
N &=-\rank (\mathcal{N})+\Drop_{0}(\mathcal{N})+\Drop_{s}(\mathcal{N})\\&+\Swan_{\infty}(\mathcal{N})+\Swan_{0}(\mathcal{N})+\Swan_{s}(\mathcal{N}),\\&\geq -2+1+1+1+d=d+1
\end{split}
\]
as we wanted.
\item[$vi)$] We start by bounding
\[
\begin{split}
M_{4}&:=\frac{1}{p^2}\sum_{a\in\F}\sum_{b\in\F}\Big|\frac{1}{\sqrt{p}}\sum_{n\in I}e_{p}(bn+a\overline{n}^{d})\Big|^{4}\\&=\frac{1}{p^{4}}\sum_{a\in\F}\sum_{b\in\F}\sum_{n_{1},n_{2},m_{1},m_{2}\in I}e_{p}(b(n_{1}+n_{2}-m_{1}-m_{2})+a(\overline{n_{1}}^{d}+\overline{n_{2}}^{d}-\overline{m_{1}}^{d}-\overline{m_{2}}^{d})).
\end{split}
\]
We use the same strategy as in \cite[page $1505$]{KoSa}: the orthogonality of the additive characters implies that
\[
M_{4}=\frac{1}{p^{3}}\sum_{a\in\F}\sum_{\substack{n_{1},n_{2},m_{1},m_{2}\in I\\n_{1}+n_{2}=m_{1}+m_{2}}}e_{p}(a(\overline{n_{1}}^{d}+\overline{n_{2}}^{d}-\overline{m_{1}}^{d}-\overline{m_{2}}^{d}))
\]
and then
\[
M_{4}=\frac{1}{p^{2}}\sum_{\substack{n_{1},n_{2},m_{1},m_{2}\in I\\n_{1}+n_{2}=m_{1}+m_{2}\\ \overline{n}_{1}^{d}+\overline{n}_{2}^{d}=\overline{m}_{1}^{d}+\overline{m}_{2}^{d}}}1.
\]
For $n_{1}+n_{2}\neq 0$, the system
\[
\begin{cases}
n_{1}+n_{2}=m_{1}+m_{2}\\
\overline{n}_{1}^{d}+\overline{n}_{2}^{d}=\overline{m}_{1}^{d}+\overline{m}_{2}^{d}
\end{cases}
\]
has at most $2d$ pairs of solutions $(m_{1},m_{2})$. On the other hand, if $n_{1}+n_{2}= 0$ then $m_{1}+m_{2}= 0$. Thus we can bound $M_4$ as
$M_{4}\ll_{d} |I|^{2}p^{-2}$. Now, by positivity we get that
\[
\frac{1}{p^2}\sum_{a\in\Fs}\sum_{b\in\Fs}\Big|\frac{1}{\sqrt{p}}\sum_{n\in I}e_{p}(bn+(a\overline{n})^{d})\Big|^{4}\leq (d,p-1)M_{4}\ll_{d} |I|^{2}p^{-2}.
\]
Choosing $|I|\leq p^{1/2+1/6}$ we obtain the result.
\end{itemize}

\subsubsection{Hyper-Kloosterman sums}
For any $r\geq 2$, let $\mathcal{K}\ell_{r}$ denote the $r$-th Kloosterman sheaf. For any $r\geq 3$ odd the family
\[
\{[x\mapsto \overline{ax}]^{*}\mathcal{K}\ell_{r}\otimes\mathcal{L}_{e_{p}(bT)}\}_{(a,b)\in\mathbb{F}_{p}^{\times}\times\mathbb{F}_{p}^{\times}}
\]
is a $2$-parameter family of $\Sp_{r+1}$-type.
\begin{itemize}
\item[$i)$] For any $(a,b)\in\Fs\times\Fs$, the sheaf $[x\mapsto\overline{ax}]^{*}\mathcal{K}\ell_{r}\otimes\mathcal{L}_{e_{p}(bT)}$ is a Fourier, irreducible middle-extension $\ell$-adic sheaf on $\overline{\mathbb{A}}_{\mathbb{F}_{p}}^{1}$ punctually pure of weight $0$. Moreover, the trace function attached to $[x\mapsto \overline{ax}]^{*}\mathcal{K}\ell_{r}\otimes\mathcal{L}_{e_{p}(bT)}$ is given by 
\[
t_{\mathcal{F}_{a,b}}:x\mapsto \Kl_{r}(\overline{ax};p)e_{p}(bx).
\]
\item[$ii)$] Thanks to \cite[Proposition $8.2$]{FKM15} and \cite[$11.0.2$]{Kat88}, one has that $c(\mathcal{F}_{a,b})\leq 5c([x\mapsto \overline{ax}]^{*}\mathcal{K}\ell_{r})^{2}c(\mathcal{L}_{e_{p}(bT)})^{2}=45c([x\mapsto \overline{ax}]^{*}\mathcal{K}\ell_{r})^{2}=45(r+3)^{2}$.
\item[$iii)$] We start with computing the Fourier transform of $t_{\mathcal{F}_{a,b}}$:
\[
\begin{split}
t_{\mathcal{K}_{a,b}}(z)&=-\frac{1}{\sqrt{p}}\sum_{x\in\Fs}\Kl_{r}(\overline{ax};p)e_{p}((b+z)x)\\&=-\frac{1}{p^{r/2}}\sum_{x\in\Fs}\Big(\sum_{x_{1},...,x_{r-1}\in\mathbb{F}_{p}^{\times}}e_{p}(x_{1}+\cdots+x_{r-1}+\overline{axx_{1}\cdots x_{r-1}})\Big)e_{p}((b+z)x)\\&=-\frac{1}{p^{r/2}}\sum_{x,x_{1},...,x_{r-1}\in\mathbb{F}_{p}^{\times}}e_{p}(x_{1}+\cdots+x_{r-1}+(b+z)x+\overline{axx_{1}\cdots x_{r-1}}).
\end{split}
\]
If $z\neq -b$, then we apply the change of variables $t=x(b+z)$ getting
\[
\begin{split}
t_{\mathcal{K}_{a,b}}(z)&=-\frac{1}{p^{r/2}}\sum_{t,x_{1},...,x_{r-1}\in\mathbb{F}_{p}^{\times}}e_{p}(x_{1}+\cdots+x_{r-1}+t+(b+z)\overline{atx_{1}\cdots x_{r-1}})\\&=\Kl_{r+1}(\overline{a}(b+z);p).
\end{split}
\]
Thus, we have that $t_{\mathcal{K}_{a,b}}(z)=\Kl_{r+1}(\overline{a}(b+z);p)=\Kl_{r+1}(\overline{a}(1+\gamma_{z}(b));p)=t_{[\gamma_{z}]^{*}\mathcal{K}_{a,1}}(b)$, where $\gamma_{z}=\begin{pmatrix}1 & z-1\\0 & 1\end{pmatrix}$.
\item[$iv)$] For any $a,b\in\mathbb{F}_{p}^{\times}$, the monodromy of $\mathcal{K}_{a,b}=[\gamma_{a,b}]^{*}\mathcal{K}\ell_{r+1}$ is the same as the one of $\mathcal{K}\ell_{r+1}$. Thus, $G_{\mathcal{K}_{a,b}}^{\text{arith}}=G_{\mathcal{K}_{a,b}}^{\text{geom}}$ and
$G_{\mathcal{K}_{a,b}}^{\text{geom}}=\Sp_{r+1}(\mathbb{C})$.
\item[$v)$]  We need to show that for any $a,z_{1},z_{2}\in \Fs$ with $z_{1}\neq z_{2}$, one has
\[
[\gamma_{z_{1}}\gamma_{z_{2}}^{-1}]^{*}\mathcal{K}_{a,1}\neq_{\text{geom}}\mathcal{K}_{a,1}\otimes\mathcal{L},
\]
for any $\ell$-adic sheaf $\mathcal{L}$ of rank $1$. This is just a consequence of (\cite[Proposition $3.6$]{FKM15b}).
\item[$vi)$] We compute
\[
\begin{split}
M_{4}&=\frac{1}{p^{2}}\sum_{a\in\Fs}\sum_{b\in\F}\Big|\frac{1}{\sqrt{p}}\sum_{n\in I}\Kl_{r}(\overline{an};p)e_{p}(bn)\Big|^{4}\\&=\frac{1}{p^{4}}\sum_{a\in\Fs}\sum_{b\in\F}\sum_{n_{1},n_{2},m_{1},m_{2}\in I}\Kl_{r}(\overline{an_{1}};p)\Kl_{r}(\overline{an_{2}};p)\Kl_{r}(-\overline{am_{1}};p)\Kl_{r}(-\overline{am_{2}};p)\times\\&\times e_{p}(b(n_{1}+n_{2}-m_{1}-m_{2})).
\end{split}
\]
By orthogonality of the additive characters, we get
\[
M_{4}=\frac{1}{p^{3}}\sum_{a\in\Fs}\sum_{\substack{n_{1},n_{2},m_{1},m_{2}\in I\\ n_{1}+n_{2}=m_{1}+m_{2}}}\Kl_{r}(\overline{an_{1}};p)\Kl_{r}(\overline{an_{2}};p)\Kl_{r}(-\overline{am_{1}};p)\Kl_{r}(-\overline{am_{2}};p).
\]
On the other hand, the sum 
\begin{equation}
\sum_{a\in\mathbb{F}_{p}^{\times}}\Kl_{r}(\overline{an_{1}};p)\Kl_{r}(\overline{an_{2}};p)\Kl_{r}(-\overline{am_{1}};p)\Kl_{r}(-\overline{am_{2}};p)
\label{eq : proklo}
\end{equation}
is of size $p$ if and only if either ($n_{1}=-n_{2}$ and $m_{1}=-m_{2}$) or ($n_{1}=m_{1}$ and $n_{2}=m_{2}$) or ($n_{1}=m_{2}$ and $n_{2}=m_{1}$), and it has size $O_{r}(\sqrt{p})$ otherwise (\cite[Corollary $3.3$]{FKM15b}). Choose $n_{1},n_{2}\in I$. We need to distinguish between two cases:
\begin{enumerate}
\item $n_{1}\neq -n_{2}$, thus we have at most two choices of $m_{1},m_{2}$ such that the sum in $(\ref{eq : proklo})$ has size $p$,
\item $n_{1}=-n_{2}$, then the sum in $(\ref{eq : proklo})$ has size $p$ for at most $|I|$ couples $(m_{1},m_{2})$.
\end{enumerate}
Thus we obtain
\[
M_{4}\ll_{r} |I|^{2}p^{-2}+|I|^{3}p^{-5/2}.
\]
Choosing $|I|\leq p^{1/2+1/8}$ we get the result.
\end{itemize}

\end{document}